\documentclass[preprint,12pt]{elsarticle}
\usepackage{graphicx} % Required for inserting images

\usepackage[utf8]{inputenc}
\usepackage[english]{babel}
\usepackage{graphicx}
\usepackage{amssymb}
\usepackage{amsmath}
\usepackage{tikz}
\usetikzlibrary{decorations.pathreplacing} %for the braces
\usepackage{commath}
\usepackage{upgreek}
\usepackage[hidelinks]{hyperref}
\usepackage{color}
\usepackage{wrapfig}
\usepackage{graphicx,wrapfig,lipsum}
\graphicspath{ {./images/} }

\usepackage{amsthm}
\newtheorem{theorem}{Theorem}[section]

\newtheorem{lemma}[theorem]{Lemma}
\newtheorem{definition}{Definition}[section]

\usepackage[a4paper, total={6.5 in, 9.25 in}]{geometry}

\date{\vspace{-5ex}}

\journal{Nonlinear Analysis}

\begin{document}

\begin{frontmatter}

\title{Weak global solvability of a doubly degenerate \\parabolic-elliptic nutrient taxis system}

\author[ucm,upb]{\corref{cor1}Federico Herrero-Hervás}
\ead{fedher01@ucm.es}
\cortext[cor1]{Corresponding author}

\affiliation[ucm]{organization={Instituto de Matemática Interdisciplinar, Departamento de Análisis Matemático y Matemática Aplicada, Universidad Complutense de Madrid},
            city={Madrid},
            postcode={28040},
            country={Spain}}

\affiliation[upb]{organization={Institut für Mathematik, Universität Paderborn},
            city={Paderborn},
            postcode={33098},
            country={Germany}}

%\affiliation[icai]{organization={Departamento de Matemática Aplicada, ICAI, Universidad Pontificia de Comillas},
 %           city={Madrid},
  %          postcode={28015},
   %         country={Spain}}

\begin{abstract}
This work studies the following doubly degenerate parabolic-elliptic nutrient taxis system
$$
\begin{cases}
    u_t = (uvu_x)_x -(u^2 vv_x)_x + uv, \\[1.5 ex]
    \hspace{0.2 cm}0 = v_{xx} - uv + f(x,t),
\end{cases}
$$
in a bounded interval $\Omega \subset \mathbb{R}$, under no-flux boundary conditions and nonnegative initial value $u(x,0) = u_0(x) \geq 0$, where $f(x,t) \geq 0$ is known external supply of the nutrient.

It is shown that for any nonnegative $u_0 \in W^{1,\infty}(\Omega)$ and $f \in C^1\big(\bar{\Omega} \times [0,\infty) \big)$, $f \not \equiv 0$, a global weak solution of the problem can be constructed by means of a regularization approach. 

The core of the analysis lies on a Harnack-type inequality for the second that allows us to overcome the lack of uniform coercivity. Together with time regularity properties, we obtain relative compactness through a combination of the Arzelà-Ascoli theorem and the Aubin-Lions lemma.
\end{abstract}


\end{frontmatter}

\section{Introduction}
Aside from random diffusion, bacterial motion can also be influenced by the presence of chemical stimuli. The effect of chemotaxis, this directed migration process in response to chemical cues, has been observed across a wide range of living systems. In bacterial populations, chemotactic responses can, under suitable conditions, give rise to intricate forms of aggregation \cite{fractal2, fractal1, BudreneBerg}. 
\\\\
From a mathematical perspective, the Keller-Segel system \cite{KS1, KS2}, which in its minimal formulation is given by
\begin{equation}\label{0.1}
      \begin{cases}
         \displaystyle \frac{\partial u}{\partial t} = \Delta u - \chi \nabla \cdot (u\nabla v), \\[1.75ex]
       \tau \displaystyle \frac{\partial v}{\partial t} = \Delta v +u - v,
    \end{cases}
\end{equation}
models such interactions between a biological species, whose population density is denoted by $u$, and a chemical agent inducing chemotaxis, with a concentration denoted by $v$. Parameter $\chi >0$ measures the chemotaxis sensitivity and $\tau\geq 0$ the diffusive timescale of the chemical. Extensive effort has been devoted to the study of system \eqref{0.1} and its closely related variants (see for instance the reviews \cite{Bellomo, Hillen, Horstmann}), both for $\tau=0$, corresponding to the assumption that the chemical diffuses significantly faster than the individuals of the species, as well as for $\tau>0$, when both diffusion processes occur on comparable timescales.
\\\\
In the case of system \eqref{0.1}, it is assumed that the substance is self-produced by the species, as a means of self-aggregation, which can lead to blow-up of solutions \cite{mah}. However, for other scenarios where the signal is consumed instead of produced ---for instance in the case of a nutrient-- comparatively less information is currently available. The corresponding mathematical formulation of such nutrient taxis process reads
\begin{equation}\label{0.2}
    \begin{cases}
         \displaystyle \frac{\partial u}{\partial t} = \Delta u - \chi\nabla \cdot (u \nabla v), \\[1.75ex]
       \tau \displaystyle \frac{\partial v}{\partial t} = \Delta v - uv.
    \end{cases}
\end{equation}
Global solutions to system \eqref{0.2} and similar variants are known to exist in two- and three-dimensional domains in fully parabolic \cite{Tao11, TaoWinkler12, Wu16, ZhangLi} and parabolic-elliptic cases \cite{TaoWinkler2019}. For this class of systems, diffusion typically dominates the long-time dynamics, resulting in stabilization toward constant steady states, even upon incorporating a growth term of the form $+uv$ to the first equation. This contrasts with experimental observations, for instance of \textit{Bacillus subtilis} colonies grown on agar plates, where complex branching patterns emerge in the presence of peptone, an attracting nutrient \cite{Fujikawa89, Fujikawa92, Fujikawa90}. Experimental information suggested that nutrient taxis plays a central role in pattern formation of such colonies, which however could not be captured by system \eqref{0.2}.
\\\\
Tied to these findings, substantial differences are found on the mathematical level when the diffusion and taxis mechanisms are considered to degenerate at small population and chemical levels. In particular, the system introduced in \cite{plaza13}, given by
\begin{equation}\label{0.3}
    \begin{cases}
         \displaystyle \frac{\partial u}{\partial t} = \nabla \cdot(u v \nabla u) - \chi \nabla \cdot (u^2 v \nabla v) + uv, \\[1.75ex]
       \displaystyle \frac{\partial v}{\partial t} = \Delta v - uv,
    \end{cases}
\end{equation}
has recently attracted significant interest. Numerical simulations carried out in \cite{plaza13} already show the emergence of branching patterns, similar to those observed in experimental studies. Particularly, the structural resemblance of these aggregates is enhanced by the taxis term in system \eqref{0.3} in contrast to the purely diffusive scenario (when $\chi =0$), which had already been studied from a numerical point of view in \cite{Kawasaki}.
\\\\
An analytical description of this process was first characterized in \cite{winkler1d}, where system \eqref{0.3} was studied on a one-dimensional domain, obtaining global weak solutions as the limit of a sequence of regularized problems. Importantly, these solutions are such that if $||v_0||_{L^1(\Omega)}$ is small enough, then 
\begin{equation}\label{0.4}
u \to u_\infty \not \equiv \text{constant} ~\text{ in } L^\infty(\Omega), ~\text{ and } ~ v\to 0 ~\text{ in } W^{1,\infty}(\Omega).   
\end{equation}
Hence, at least in nutrient-poor environments, system \eqref{0.3} is capable of generating spatially heterogeneous limits in the $u$ component, providing a paradigm for pattern formation in such scenarios. The key role played by the size of $v_0$ is directly linked to the degenerate nature of the diffusive and taxis terms. If the initial nutrient concentration is small enough, it is then rapidly consumed by the species, reaching negligible levels before diffusion homogenizes the population over the domain. 
\\\\
Since this first study of system \eqref{0.3}, many other results have become available. Still in dimension one, in \cite{li-wink} the assumptions made on the initial data made in \cite{winkler1d} were weakened. Particularly, if allowed for compactly supported $u_0$, a relevant case for applications. Aside, in \cite{HH-Cauchy} the corresponding one-dimensional Cauchy problem was proved to admit globally defined weak solutions upon sufficient regularity for $u_0$, $v_0$ and $v_{0x}$. On two-dimensional domains, weak global solvability was also proved in \cite{winkler2d}, again under appropriate smallness hypothesis for $v_0$. Global $L^\infty$ bounds were later obtained under such hypotheses in \cite{winkler-l-inf-2d}. This smallness requirement was, however, recently been removed in \cite{ZL1}, provided that $(u_0,v_0) \in W^{1, \infty}(\Omega)$ are such that $\int_\Omega \ln u_0 > - \infty$ and $v_0>0$, proving a convergence result similar to \eqref{0.4}, although with weak$^*$ convergence for $u$. Analogous convergence results have also been recently obtained in \cite{Duan} for a slight variation of system \eqref{0.3}.
\\\\
Further extensions include logistic dynamics for the bacterial species, which were first investigated for superquadratic degradation terms of the form $\lambda u - \mu u^\alpha$ for some $\alpha>2$ in \cite{logistic}, and more recently have been extended to the classical logistic case with $\alpha = 2$ in \cite{winkler-li-logist}. Moreover, in \cite{yuting} system \eqref{0.3} is extended to two-species, with corresponding logistic and competitive dynamics.
\\\\
These results provide a comprehensive description of the dynamics of system \eqref{0.3} and related variants when the diffusion timescales of the species and the nutrient are of comparable magnitude. In this work, however, we investigate the effects of a fast diffusive nutrient by approximating the second equation by its elliptic counterpart. We also include an external supply of the nutrient by means of a known source function $f(x,t) \geq 0$, considering the system posed over a bounded one-dimensional domain $\Omega \subset \mathbb{R}$. In this way, the problem reads
\begin{equation}\label{1.1}
    \begin{cases}
        \displaystyle \frac{\partial u}{\partial t} = \frac{\partial}{\partial x}\left(u v \frac{\partial u}{\partial x}\right) - \frac{\partial}{\partial x}\left(u^2 v \frac{\partial v}{\partial x}\right) + u v, & x\in \Omega, ~t>0, \\[2ex]
       \displaystyle \hspace{0.2 cm}0 \hspace{0.15 cm}=  \frac{\partial^2 v}{\partial x^2} - u v + f(x,t), & x\in \Omega, ~t>0, \\[1.75ex]
       \displaystyle uv\frac{\partial u}{\partial x} - u^2 v\frac{\partial v}{\partial x} = \frac{\partial v}{\partial x} = 0, & x\in \partial\Omega, ~t>0 \\[1.75ex]
       u(x,0) = u_0(x) \geq 0,   & x\in \Omega.
    \end{cases}
\end{equation}
In contrast to the fully parabolic version, several differences arise in the analysis of system \eqref{1.1}. In particular, the fact that $u$ may not be strictly positive over $\Omega$ results in the loss of coercivity in the elliptic equation. Moreover, in the parabolic version, an $L^\infty$ bound for $v$ could be readily obtained through a maximum principle, that no longer becomes applicable in this scenario.
\\\\
As in other works comprising system \eqref{0.3}, we consider regularized variants of system \eqref{1.1} by taking a parameter $\varepsilon \in (0,1)$ and solving the corresponding system with $u_\varepsilon(x,0) = u_0(x) + \varepsilon >0$. Our approach for the study of the second equation lies on a Harnack-type inequality that can be obtained thanks to the one-dimensional setting. Moreover, exploiting the fact that $\int_\Omega f = \int_\Omega u_\varepsilon v_\varepsilon$, time-dependent $L^\infty$ bounds can be obtained for $v_\varepsilon$.
\\\\
\textbf{Main Results:} \quad For our case, system \eqref{1.1} is considered under the following hypotheses
\begin{align}
    &u_0(x)\in W^{1,\infty}(\Omega)  ~\text{ is nonnegative,}\label{h1}\\[1.5 ex]
    &f(x,t) \in  C^1\big(\bar{\Omega} \times [0,\infty)\big), ~\text{is nonnegative, with} ~f \not \equiv 0. \label{h2}
\end{align} 
This allows us to prove the following main result. 
\begin{theorem}\label{t1}
    Assume \eqref{h1} and \eqref{h2} hold. Then, there exists a global weak solution $(u,v)$ to system \eqref{1.1} in the sense of Definition \ref{weak-sol} below, which moreover satisfies
    $$
      \begin{cases}
        u \in L^\infty_{\text{loc}}\big((0,\infty); L^p(\Omega)\big) ~\text{ for all } p \geq 1, \\[1.5 ex]
        v \in C^0_\text{loc}(\bar{\Omega} \times [0,\infty)\big) \cap L^\infty_\text{loc}\big((0,\infty); C^{1,\alpha}(\Omega)\big),
    \end{cases}
    $$
\end{theorem}
Condition \eqref{h1} is typically assumed for the fully parabolic versions, with the control of $u_0$ and $(u_0)_x$ in $L^\infty$ allowing to establish local existence of solutions. With respect to the hypotheses for $f$, the requirement $f\not \equiv 0$ prevents us from directly obtaining $v \equiv 0$ as a solution to the second equation, and thus obtain meaningful dynamics. Moreover, $C^1$ regularity allows us to obtain solutions to the second equation that are continuous in time, and at least on the regularized level with a time derivative that is also continuous in time. The same assumption was used in \cite{TaoWinkler2019} for a nondegenerate parabolic-elliptic nutrient taxis system similar to ours.
\\\\ 
\textbf{Structure of the work:} \quad After this introduction, in Section \ref{s2}, the concept of weak solutions to system \eqref{1.1} is introduced in Definition \ref{weak-sol}, as well as the regularization approach followed. The regularized system, depending on a parameter $\varepsilon\in (0,1)$ is given by \eqref{2.1}, with Lemma \ref{l2.1} ensuring local existence of classical solutions for each $\varepsilon$ over a maximal interval $(0,T_{\max,\varepsilon})$. Moreover, $L^p$ bounds for $u_\varepsilon$ can be readily obtained in Lemma \ref{l2.3}, exploiting the structure of $\Delta v_\varepsilon$. The absence of the time derivative $v_{\varepsilon t}$ makes the analysis comparatively simpler than in the fully parabolic counterpart. 
\\\\
In contrast, spatial regularity properties for $v_\varepsilon$, to which Section \ref{s3} is devoted, become more involved. The core of the analysis relies on a Harnack-type inequality proved in Lemma \ref{l3.2}. The assumption on $\Omega$ being a convex and bounded one-dimensional domain ---that is, an interval--- which we consider to be of the form $ \Omega = (0,L)$, $L>0$ plays a crucial role on it. In particular, the proof is based on integrating the following identity, that holds for any positive real-valued function $g$ of class $C^2$ 
$$
\frac{g''}{g} = \left( \frac{g'}{g}\right)' + \left(\frac{g'}{g} \right)^2.
$$
This eventually leads to time-dependent $L^\infty$ bounds for $v_\varepsilon$ and $v_{\varepsilon x}$, obtained in Lemma \ref{l3.3} and Lemma \ref{l3.4}, respectively. 
\\\\
Next, Section \ref{s4} studies time regularity properties. By first proving a bound for $u_{\varepsilon t}$ in $L^2\big((0,T); (W^{1,2}(\Omega))^*\big)$ in Lemma \ref{l4.2} for any $T\in(0,T_{\max,\varepsilon})$, Hölder continuity in time for $v_{\varepsilon}$ is obtained. After some auxiliary steps, the result is proved in Lemma \ref{l4.5}, based on considering the forward differences $\delta_h v_\varepsilon(x,t) := v_\varepsilon(x,t+h)- v_\varepsilon(x,t)$, for $h>0$. Although this suffices for obtaining relative compactness of $(v_\varepsilon)_{\varepsilon \in (0,1)}$, with respect to $(u_\varepsilon)_{\varepsilon \in (0,1)}$, the lack of direct control over the spatial derivative $u_{\varepsilon x}$ forces us to extend the analysis to the auxiliary sequence $(u_\varepsilon^{(p+1)/2})_{\varepsilon \in (0,1)}$, for $p > 1$. To this end, in Lemma \ref{l4-aux}, a bound for $\partial_t\big(u_\varepsilon^{(p+1)/2}\big)$ in $L^1\big((0,T); (W^{3,2}(\Omega)^*)$ is derived.  
\\\\
Lastly, Section \ref{s5} begins with the proof that $T_{\max,\varepsilon} = \infty$ for all $\varepsilon \in (0,1)$ in Lemma \ref{l5.1}, showing that all regularized solutions exist globally in time. Lastly, we prove Lemma \ref{l5.2}, which allows us to extract a convergent subsequence $(\varepsilon_j)_{j\in\mathbb{N}} \subset (0,1)$ and limit functions $u$ and $v$ that form a weak solution to system \eqref{1.1}. The article thus finishes with the proof of Theorem \ref{t1}.
\\\\
An analytical characterization of the long-time behavior of solutions to system \eqref{1.1} depending on $f$ is however still to be obtained. 

\section{Regularized problems and $L^p$ estimates for $u_\varepsilon$}\label{s2}
\subsection{Weak solutions and regularized problems}
Due to the double degeneracy of the diffusion mechanism in \eqref{1.1}, we confine the analysis in this work to the existence of global weak solutions, as defined below.
\begin{definition}\label{weak-sol}
Let $u$  and $v$ be nonnegative functions defined on $\mathbb{R} \times (0,\infty)$ satisfying
\begin{equation}\label{2.01}
    \begin{cases}
        u \in L^1_\text{loc} \left(\Omega \times [0, \infty)\right), \\[1.5ex]
        v \in L^\infty_\text{loc}\left(\Omega \times [0, \infty)\right) \cap L^1_\text{loc} \left([0, \infty); W^{1,1}(\Omega)\right),
    \end{cases}
\end{equation}
and 
\begin{equation}\label{2.02}
   u^2v, ~u^2 v_x \in L^1_\text{loc}\big(\Omega \times (0,\infty)\big)
\end{equation}
then, $(u,v)$ will be called a weak solution to system \eqref{1.1} if 
\begin{equation}\label{2.03}
\begin{split}
     \int_0^\infty \int_\Omega u \varphi_t + \int_\Omega u_0 \varphi(\cdot, 0) =  -\frac{1}{2}\int_0^\infty \int_\Omega u^2 v_x \varphi_x - \frac{1}{2}\int_0^\infty \int_\Omega u^2 v \varphi_{xx} \\[1.5ex]
     - \int_0^\infty \int_\Omega u^2 v v_x \varphi_x  - \int_0^\infty \int_\Omega u v \varphi,
\end{split}
\end{equation}
and
\begin{equation}\label{2.04}
    \int_\Omega v_x(\cdot,t) \varphi_x(\cdot,t) +\int_\Omega u(\cdot,t) v(\cdot,t) \varphi(\cdot,t) = \int_\Omega f(\cdot,t) \varphi(\cdot,t), \quad \text{for a.e. } t >0.
\end{equation}
for all $\varphi \in C_0^\infty (\bar{\Omega} \times [0, \infty))$. 
\end{definition}

To be able to obtain weak solutions in the sense specified, we consider regularized versions of system \eqref{1.1}, for $\varepsilon \in (0,1)$. 
\begin{equation}\label{2.1}
    \begin{cases}
        \displaystyle \frac{\partial u_\varepsilon}{\partial t} = \frac{\partial}{\partial x}(u_\varepsilon v_\varepsilon u_{\varepsilon x}) - \frac{\partial}{\partial x}(u_\varepsilon^2 v_\varepsilon v_{\varepsilon x}) + u_\varepsilon v_\varepsilon, & x\in \Omega, ~t>0, \\[1.75ex]
       \displaystyle     \displaystyle 0 =  \frac{\partial^2 v_\varepsilon}{\partial x^2} - u_\varepsilon v_\varepsilon + f(x,t), & x\in \Omega, ~t>0, \\[1.5ex]
       \displaystyle \frac{\partial u_\varepsilon}{\partial x} = \displaystyle \frac{\partial v_\varepsilon}{\partial x}  = 0, & x\in \partial\Omega, ~t>0 \\[1.5ex]
       u_\varepsilon(x,0) = u_0(x) + \varepsilon > 0,   & x\in \Omega,
    \end{cases}
\end{equation}
The strict positivity of the initial value $u_\varepsilon(x,0)$ for each fixed $\varepsilon \in (0,1)$ removes the degeneracy of the diffusion process, and thus allows us to construct classical solutions to the regularized system \eqref{2.1}. 
\begin{lemma}\label{l2.1}
    Assume that $u_0$ and $f$ satisfy \eqref{h1} and \eqref{h2}. Then, for any $\varepsilon \in (0,1)$, there exists $T_{\max, \varepsilon} \in (0,\infty]$ and a uniquely determined pair of functions $(u_\varepsilon, v_\varepsilon)$ satisfying
    \begin{equation}\label{2.05}
        \begin{cases}
            u_\varepsilon \in C^0\big(\bar{\Omega} \times [0,T_{\max,\varepsilon})\big) \cap C^{2,1}\big(\Omega \times[0,T_{\max,\varepsilon})\big) \\[1.5 ex]
            v_\varepsilon \in C^{2,0}\big(\bar{\Omega} \times [0,T_{\max,\varepsilon})\big)
        \end{cases}
    \end{equation}
    with $u_\varepsilon >0$ and $v_\varepsilon \geq 0$ in $\bar{\Omega} \times (0,T_{\max,\varepsilon})$, that solve the regularized system \eqref{2.1} in the classical sense in $\Omega \times [0,T_{\max,\varepsilon})$, with the property that
    \begin{equation}\label{2.06}
        \text{if } T_{\max,\varepsilon} <\infty, \text{ then } \limsup_{t \nearrow T_{\max,\varepsilon}} ||u_\varepsilon(\cdot,t)||_{L^\infty(\Omega)} = \infty.
    \end{equation}
\end{lemma}
\begin{proof}
    Given the parabolic-elliptic nature of the regularized system \eqref{2.1}, we consider a fixed point approach. We describe the main steps of the proof and refer the readers to \cite{TaoWinkler2019} for a more detailed derivation for a similar nondegenerate nutrient taxis system. 
    \\\\
    For a fixed $\varepsilon \in (0,1)$, we introduce
    $$
    m_\varepsilon := \frac{1}{2} \min\{1,\inf_{x\in\Omega} u_\varepsilon(x,0)\} = \frac{\varepsilon}{2}>0~ \text{ and } ~M_\varepsilon := 2 ||u_\varepsilon(x,0)||_{L^\infty(\Omega)}>m_\varepsilon,
    $$
    which allows us to define the following constants
    \begin{equation}\label{AB}
     \begin{split}
          A_\varepsilon &:= M_\varepsilon \cdot ||f||^2_{L^\infty(\Omega \times (0,1))}\cdot\left(|\Omega|^2 \left(1+\frac{M_\varepsilon}{m_\varepsilon}\right)^2 + \frac{M_\varepsilon}{m_\varepsilon^2}\right) >0,
        \\[1.5 ex]
        B_\varepsilon &:= \frac{M_\varepsilon ||f||_{L^\infty(\Omega \times (0,1))}}{m_\varepsilon} \left( 1+ M_\varepsilon ||f||_{L^\infty(\Omega \times (0,1))} \right) >0.
     \end{split}        
    \end{equation}
    Through $A_\varepsilon$ and $B_\varepsilon$, we take
    \begin{equation}\label{2.06.2}
        T_\varepsilon := \min \left\{\frac{\ln 2}{A_\varepsilon}, \frac{||u_\varepsilon (x,0)||_{L^\infty(\Omega)}}{B_\varepsilon}, 1\right\},
    \end{equation}
    satisfying $T_\varepsilon >0$. Next, we consider the closed convex subset
    $$
    S:= \left \{\varphi \in X, m_\varepsilon \leq \varphi \leq M_\varepsilon \text{ in } \bar{\Omega} \times [0,T_\varepsilon]   \right\}.
    $$
    of the Banach space $X:= C^{\theta,\theta/2}(\bar{\Omega}\times[0,T_\varepsilon])$ for some $\theta \in (0,1)$. 
    \\\\
    Taking an arbitrary $\bar{u}_\varepsilon \in S$, for all $t\in (0,T_\varepsilon)$, the problem
    \begin{equation}\label{2.07}
        \begin{cases}
            \displaystyle- \frac{\partial^2 v_{\varepsilon}}{\partial x^2} + \bar{u}_\varepsilon(x,t) v = f(x,t), & x \in \Omega, \\[1.75 ex]
            \displaystyle \hspace{0.35 cm}\frac{\partial v_\varepsilon}{\partial x} = 0, & x \in \partial\Omega,
        \end{cases}
    \end{equation}
    admits a unique weak solution $v_\varepsilon(\cdot,t) \in W^{1,2}(\Omega)$ due to the positivity of $\bar{u}_\varepsilon$, that moreover belongs to $C^{2+\theta}(\bar{\Omega})$ satisfying \eqref{2.07} in the classical sense, and as $T_\varepsilon<1$, we have
    \begin{equation}\label{2.08}
        0 \leq v_\varepsilon(\cdot, t) \leq \frac{||f||_{L^\infty(\Omega \times (0,1))}}{m_\varepsilon} \quad \text{for all } t\in (0,T_\varepsilon).
    \end{equation}
    This directly results in $||v_{\varepsilon xx}(\cdot,t)||_{L^\infty(\Omega)} \leq ||f||_{L^\infty(\Omega \times (0,1))} \left(1+\frac{M_\varepsilon}{m_\varepsilon}\right)$, and a direct integration, making use of the one-dimensional setting yields 
    \begin{equation}\label{2.09}
        ||v_{\varepsilon x}(\cdot,t)||_{L^\infty(\Omega)}\leq |\Omega| \cdot ||f||_{L^\infty(\Omega \times (0,1))} \left(1+\frac{M_\varepsilon}{m_\varepsilon}\right) \quad \text{for all } t\in (0,T_\varepsilon).
    \end{equation}
    It can be checked following \cite{TaoWinkler2019} that $v_\varepsilon$ enjoys further time regularity properties. In particular, $v_\varepsilon \in C^{1+\theta_1,\theta_1}(\bar{\Omega} \times [0,T_\varepsilon])$ for some $\theta_1 =\theta_1(\bar{u}_\varepsilon) \in (0,1)$.
    \\\\
    With $\bar{u}_\varepsilon \in S$ and $v_\varepsilon \in C^{1+\theta_1,\theta_1}(\bar{\Omega} \times [0,T])$, we next turn to the parabolic equation. In particular, the problem
    \begin{equation}\label{2.010}
        \begin{cases}
            \displaystyle u_{\varepsilon t} -\left( \bar{u}_\varepsilon v_\varepsilon u_{\varepsilon x} \right)_x  + 2 \bar{u}_\varepsilon v_\varepsilon v_{\varepsilon x} \cdot u_{\varepsilon x} 
            \\[1.5 ex] \hspace{0.6 cm} = -u_\varepsilon \big( \bar{u}_\varepsilon |v_{\varepsilon x}|^2\big)  - u_\varepsilon^2 \big(\bar{u}_\varepsilon^2 v_\varepsilon^2\big) + \bar{u}_\varepsilon^2 v_\varepsilon f(x,t) + \bar{u}_\varepsilon v_\varepsilon, & x \in \Omega, ~t \in (0,T_\varepsilon), \\[1.75 ex]
            \displaystyle u_{\varepsilon x} = 0 & x \in \partial  \Omega, ~t \in (0,T_\varepsilon), \\[1.5 ex]
            u_\varepsilon(x,0) = u_0(x) + \varepsilon & x \in \Omega,
        \end{cases}
    \end{equation}
    has a unique classical solution $u_\varepsilon \in C^0(\bar{\Omega} \times [0,T_\varepsilon]) \cap C^{2,1}(\bar{\Omega} \times (0,T_\varepsilon))$. The main part of the fixed point argument is to use bounds \eqref{2.08} and \eqref{2.09} for $v_\varepsilon$, combined with the fact that $m_\varepsilon\leq \bar{u}_\varepsilon \leq M_\varepsilon$ and the definition of $T_\varepsilon$ in \eqref{2.06.2} to prove that $u_\varepsilon$, the solution to \eqref{2.010}, belongs again to $S$.
    \\\\
    Firstly, a subsolution can be built using the fact that for all $t \in (0,T_\varepsilon)$
    $$
    - \big( \bar{u}_\varepsilon |v_{\varepsilon x}|^2\big) \geq -M_\varepsilon \cdot|\Omega|^2 \cdot ||f||^2_{L^\infty(\Omega \times (0,1))} \left(1+\frac{M_\varepsilon}{m_\varepsilon}\right)^2, ~\text{ and }~- \big(\bar{u}_\varepsilon^2 v_\varepsilon^2\big) \geq - \frac{M_\varepsilon^2}{m_\varepsilon^2}||f||^2_{L^\infty(\Omega \times (0,1))}.
    $$
    In this way, $u_\varepsilon(x,t) \geq \underline{y}_{\hspace{0.05 cm}\varepsilon}(t)$ for all $(x,t) \in \bar{\Omega} \times [0,T_\varepsilon]$, where $\underline{y}_{\hspace{0.05 cm}\varepsilon}$ is the solution to 
    $$
    \begin{cases}
        \displaystyle \underline{y}_{\hspace{0.05 cm}\varepsilon}'(t) = -\left[M_\varepsilon \cdot|\Omega|^2 \cdot ||f||^2_{L^\infty(\Omega \times (0,1))} \left(1+\frac{M_\varepsilon}{m_\varepsilon}\right)^2 \right] \underline{y}_{\hspace{0.05 cm}\varepsilon} - \left[\frac{M_\varepsilon^2}{m_\varepsilon^2}||f||^2_{L^\infty(\Omega \times (0,1))}\right] \underline{y}_{\hspace{0.05 cm}\varepsilon}^2, & t>0 \\[1.5ex]
        \underline{y}_{\hspace{0.05 cm}\varepsilon}(0) = \varepsilon
    \end{cases}
    $$
    As $\underline{y}_{\hspace{0.05 cm}\varepsilon}'(t) \leq 0$ and $\underline{y}_{\hspace{0.05 cm}\varepsilon}(0) = \varepsilon <1$, one can easily conclude that 
    $$
    \underline{y}_{\hspace{0.05 cm}\varepsilon}(t) \geq \varepsilon e^{-A_\varepsilon t}  \geq \varepsilon e^{-A_\varepsilon T_\varepsilon} \geq \frac{\varepsilon}{2} = m_\varepsilon, \quad \text{for all } t \in (0,T_\varepsilon),
    $$
    where we recall that $A_\varepsilon$ was given in \eqref{AB} and $T_\varepsilon$ was chosen in \eqref{2.06.2}.
    \\\\
    Similarly, a supersolution can be built, taking into account that
    $$
    \bar{u}_\varepsilon^2 v_\varepsilon f(x,t) + \bar{u}_\varepsilon v_\varepsilon \leq \frac{M_\varepsilon^2 ||f||^2_{L^\infty(\Omega \times (0,1))}}{m_\varepsilon} + \frac{M_\varepsilon ||f||_{L^\infty(\Omega \times (0,1))}}{m_\varepsilon} = B_\varepsilon.
    $$
    Thus, $u_\varepsilon(x,t) \leq \overline{y}_\varepsilon(t)$, the solution to
    $$
    \begin{cases}
        \overline{y}_\varepsilon'(t) = B_\varepsilon, & t>0, \\[1.5 ex]
        \overline{y}_\varepsilon(0) = ||u_\varepsilon(x,0)||_{L^\infty(\Omega)},
    \end{cases}
    $$
    which satisfies
    $$
     \overline{y}_\varepsilon(t) = ||u_\varepsilon(x,0)||_{L^\infty(\Omega)} + B_\varepsilon\cdot t \leq ||u_\varepsilon(x,0)||_{L^\infty(\Omega)} + B_\varepsilon\cdot T_\varepsilon \leq 2||u_\varepsilon(x,0)||_{L^\infty(\Omega)} = M_\varepsilon,
    $$
    for all $t \in (0,T_\varepsilon)$. It can lastly be seen that there exists $c>0$ such that $||u_\varepsilon||_{C^{\theta_2, \theta_2/2}(\bar{\Omega}\times[0,T_\varepsilon]} \leq c$ for some $\theta_2 \in (0,1)$ with $\theta<\theta_2$. Thus by defining $F \bar{u}_\varepsilon := u_\varepsilon$, then $F$ maps $S$ onto itself. 
    \\\\
    It can then be seen by Schauder's theorem (see again \cite{TaoWinkler2019} for the details) that $F$ has a fixed point $u_\varepsilon \in S$, which, along with $v_\varepsilon$ computed pointwise in time as the solution to the elliptic equation, provides a classical solution to system \eqref{2.1} in $\Omega \times (0,T)$, which can be extended to some maximal $T_{\max,\varepsilon} \in (0,\infty]$ with the property \eqref{2.06}. Uniqueness is standard.
\end{proof}

\subsection{Time-dependent $L^p$ estimates for $u_\varepsilon$}
Having introduced the regularized variants of system \eqref{1.1} in \eqref{2.1}, we first obtain some basic a priori estimates for $u_\varepsilon$. The elliptic nature of the $v_\varepsilon$ equation allows for a straightforward estimate of $||u_\varepsilon||_{L^p(\Omega)}$ for arbitrary but finite $p \geq 1$, by expanding the cross-diffusive term and substituting the expression for $v_{\varepsilon xx}$ given in the second equation of \eqref{2.1}.
\\\\
To begin with, a direct integration results in the following basic estimate for $||u||_{L^1(\Omega)}$.
\begin{lemma}\label{l2.2}
   Assume \eqref{h1} and \eqref{h2}. Then for all $\varepsilon>0$ and $t \in (0,T_{\max, \varepsilon})$, $u_\varepsilon$ satisfies
    $$
    \int_\Omega u_0 \leq \int_\Omega u_\varepsilon(\cdot,t) \leq |\Omega| \cdot \big(1 + ||f||_{L^\infty(\Omega \times (0,t))} \cdot t \hspace{0.05 cm} \big) + \int_\Omega u_0.
    $$
\end{lemma}
\begin{proof}
    Integrating the first equation in \eqref{2.1} over $\Omega$ and considering the nonnegativity of the regularized solutions obtained in Lemma \ref{l2.1} leads to
    $$
    \frac{d}{dt} \int_\Omega u_\varepsilon(\cdot,t) = \int_\Omega u_\varepsilon(\cdot,t) v_\varepsilon(\cdot,t) \geq 0, \quad \text{for all } t \in (0,T_{\max, \varepsilon}).
    $$
    Hence the lower bound for $ \int_\Omega u_\varepsilon$ is directly obtained, since 
    $$
    \int_\Omega u_\varepsilon(\cdot,t)  \geq \int_\Omega u_\varepsilon(\cdot,0) = \int_\Omega (u_0 + \varepsilon) \geq \int_\Omega u_0, \quad \text{for all } t \in (0,T_{\max, \varepsilon}).
    $$
    Similarly, integrating the second equation, we arrive at
    $$ 
    \int_\Omega u_\varepsilon(\cdot,t) v_\varepsilon(\cdot,t)  = \int_\Omega f(\cdot, t) \leq ||f||_{L^\infty(\Omega \times (0,t))} \cdot |\Omega| , \quad \text{for all } t \in (0,T_{\max, \varepsilon}).
    $$
    Thus, as
    $$
     \frac{d}{dt} \int_\Omega u_\varepsilon(\cdot,t) = \int_\Omega u_\varepsilon(\cdot,t) v_\varepsilon(\cdot,t) \leq ||f||_{L^\infty(\Omega \times (0,t))} \cdot |\Omega|,
    $$
    and as $\varepsilon < 1$, a time integration provides
    \begin{equation*}
        \begin{split}
    \int_\Omega u_\varepsilon(\cdot,t) &\leq |\Omega|\cdot ||f||_{L^\infty(\Omega \times (0,t))} \cdot t + \int_\Omega u_\varepsilon(\cdot,0) \\[1.5 ex]
    &=  |\Omega|\cdot ||f||_{L^\infty(\Omega \times (0,t))} \cdot t + \int_\Omega (u_0 + \varepsilon) \\[1.5 ex]\displaystyle & \leq  |\Omega| \cdot \left( 1+ ||f||_{L^\infty(\Omega \times (0,t))} \cdot t \hspace{0.05 cm}\right) + \int_\Omega u_0, \quad \text{for all } t \in (0,T_{\max, \varepsilon}),
        \end{split}
    \end{equation*}
    which finishes the proof.    
\end{proof}
Next, we prove a corresponding time-dependent bound for $||u_\varepsilon||_{L^p(\Omega)}$ for $p>1$. Besides, another relevant estimate is obtained, which proves useful for obtaining further time regularity properties of the regularized solutions in Section \ref{s4}.
\begin{lemma}\label{l2.3}
    Let \eqref{h1} and \eqref{h2} hold. Then, for any $\varepsilon \in (0,1)$, $p>1$ and $T \in (0,T_{\max,\varepsilon})$, there exists $C(p,T)>0$ independent of $\varepsilon$ such that
    $$
    \int_\Omega u_\varepsilon^p(\cdot,t) + \int_0^T \int_\Omega u_\varepsilon^{p-1}(\cdot,t)  v_\varepsilon(\cdot,t)  |u_{\varepsilon x}(\cdot,t)|^2  \leq C(p,T), \quad \text{for all } t \in (0,T).
    $$
\end{lemma}
\begin{proof}
    Testing the first equation in \eqref{2.1} by $u_\varepsilon^{p-1}$ and integrating by parts, we obtain
    \begin{equation}\label{2.2}
        \begin{split}
                \frac{1}{p} \frac{d}{dt} \int_\Omega u_\varepsilon^p &= -(p-1)\int_\Omega u_\varepsilon^{p-2} u_{\varepsilon x} \big(u_\varepsilon v_\varepsilon u_{\varepsilon x} - u_\varepsilon^2 v_\varepsilon v_{\varepsilon x}\big) + \int_\Omega u_\varepsilon^p v_\varepsilon
                \\[1.5 ex]
                &=-(p-1) \int_\Omega u_\varepsilon^{p-1} v_\varepsilon |u_{\varepsilon x}|^2 + (p-1) \int_\Omega u_\varepsilon^p v_\varepsilon u_{\varepsilon x} v_{\varepsilon x}, \quad \text{for all } t \in (0,T_{\max, \varepsilon}).
        \end{split}
    \end{equation}
    The main idea is to integrate by parts once again again the final term on the right-hand side. In particular, as it involves $v_{\varepsilon x}$, it allows us to reach an expression depending on $v_{\varepsilon xx}$, where due to elliptic nature of the second equation in \eqref{2.1}, we can directly substitute it by $u_\varepsilon v_\varepsilon - f$. 
    \\\\
    Moreover, as $\displaystyle u_\varepsilon^p u_{\varepsilon x} = \frac{1}{p+1} (u_\varepsilon^{p+1})_x$, using Young's inequality we have
    \begin{equation}\label{2.3}
        \begin{split}
           (p-1) \int_\Omega u_\varepsilon^p v_\varepsilon u_{\varepsilon x} v_{\varepsilon x} &= - \frac{p-1}{p+1}\int_\Omega u_\varepsilon^{p+1} (v_\varepsilon v_{\varepsilon x})_x =  - \frac{p-1}{p+1}\int_\Omega u_\varepsilon^{p+1} v_\varepsilon |v_{\varepsilon x}|^2 
           \\[1.5 ex]
           & - \frac{p-1}{p+1}\int_\Omega u_\varepsilon^{p+1} v_\varepsilon v_{\varepsilon xx} \leq - \frac{p-1}{p+1}\int_\Omega u_\varepsilon^{p+1} v_\varepsilon (u_\varepsilon v_\varepsilon - f) \\[1.5 ex]
           &\leq - \frac{p-1}{p+1}\int_\Omega u_\varepsilon^{p+2} v_\varepsilon^2 + \frac{p-1}{p+1} ||f||_{L^\infty(\Omega \times (0,t))} \int_\Omega u_\varepsilon^{p+1} v_\varepsilon \\[1.5 ex]
           &\leq - \frac{p-1}{p+1}\int_\Omega u_\varepsilon^{p+2} v_\varepsilon^2 + \eta \int_\Omega u_\varepsilon^{p+2} v_\varepsilon^2 + c(\eta) \int_\Omega u_\varepsilon^p,
        \end{split}
    \end{equation}
    for some $\eta>0$, and all $t \in (0,T_{\max, \varepsilon})$, where the last term on the first line was dropped due to its nonpositivity. 
    \\\\
    Hence, by combining \eqref{2.2} and \eqref{2.3}, we arrive at
    $$
    \frac{1}{p} \frac{d}{dt} \int_\Omega u_\varepsilon^p + (p-1) \int_\Omega u_\varepsilon^{p-1} v_\varepsilon |u_{\varepsilon x}|^2 \leq \left(\eta - \frac{p-1}{p+1}\right) \int_\Omega u_\varepsilon^{p+2} v_\varepsilon^2 + c(\eta) \int_\Omega u_\varepsilon^p + \int_\Omega u_\varepsilon^p v_\varepsilon,
    $$
    for all $t \in (0,T_{\max, \varepsilon})$. For the last term, using Young's inequality twice, for any given $\tilde{\eta}>0$, there exists $\tilde{c}(\tilde{\eta})>0$ such that
    \begin{equation*}
        \begin{split}            
     \int_\Omega u_\varepsilon^p v_\varepsilon &\leq \tilde{\eta} \int_\Omega u_\varepsilon^{p+2} v_\varepsilon^2 + \tilde{c}(\tilde{\eta}) \int_\Omega u_\varepsilon^{p-2}\\[1.5 ex]
     &\leq \tilde{\eta} \int_\Omega u_\varepsilon^{p+2} v_\varepsilon^2  +\tilde{c}(\tilde{\eta})\cdot \left(\frac{2}{p} \cdot   |\Omega| + \frac{p-2}{p} \int_\Omega u_\varepsilon^p  \right), \quad \text{for all } t \in (0,T_{\max, \varepsilon}).
        \end{split}
    \end{equation*}
    In particular, if for a given $p>1$, $\eta$ and $\tilde{\eta}$ are chosen such that $\eta + \tilde{\eta} < \frac{p-1}{p+1}$, then we obtain
    \begin{equation}\label{2.4}
    \begin{split}
        \frac{1}{p} \frac{d}{dt} \int_\Omega u_\varepsilon^p &+(p-1) \int_\Omega u_\varepsilon^{p-1} v_\varepsilon |u_{\varepsilon x}|^2 + \left(\frac{p-1}{p+1}-\eta - \tilde{\eta} \right) \int_\Omega u_\varepsilon^{p+2} v_\varepsilon^2 \\[1.5 ex]
       & \leq \left(c(\eta) + \tilde{c}(\tilde{\eta}) \cdot \frac{p-2}{p} \right)\int_\Omega u_\varepsilon^p + \frac{2 \tilde{c}(\tilde{\eta})}{p} |\Omega|,      \quad \text{for all }t \in (0,T_{\max, \varepsilon}).
    \end{split}
    \end{equation}
    In particular, for $\int_\Omega u_\varepsilon^p$, this results in the differential inequality
    $$
    \frac{1}{p}\frac{d}{dt} \int_\Omega u_\varepsilon^p \leq c_1(p) \left( 1+ \int_\Omega u_\varepsilon^p \right),
    $$
    providing for any $T\in (0,T_{\max,\varepsilon})$ a $c(p,T)>0$ such that $\int_\Omega u_\varepsilon^p(\cdot,t) \leq c(p,T)$ for all $t \in (0,T)$. A further time integration of \eqref{2.4} results in the bound for $\int_0^T \int_\Omega u_\varepsilon^{p-1} v_\varepsilon |u_{\varepsilon x}|^2$, finishing the proof.
\end{proof}

\section{A Harnack inequality and elliptic regularity for $v_\varepsilon$}\label{s3}
Having established $L^p(\Omega)$ bounds for $p>1$ for $u_\epsilon$ in Lemma \ref{l2.3}, we now turn to obtaining suitable estimates for $v_\varepsilon$. It is precisely the lack of a strictly positive lower bound for $u(x,0)$ what makes the analysis of the second equation in system \eqref{2.1} particularly challenging. In particular, the lack of an $\varepsilon-$independent positive lower bound for $u_\varepsilon $ leads to the loss of coercivity of the operator $-\frac{\partial^2}{\partial x^2} + u_\varepsilon$. 
\\\\
It is precisely at this stage where we rely on the one-dimensional setting, as it allows us to rewrite the equation using a certain differential inequality based on the derivatives of the logarithm. Namely, we make use of the following identity that holds for any positive real-valued function $g: \Omega \subset \mathbb{R} \to \mathbb{R}^+$ of class $C^2(\Omega)$
\begin{equation}\label{3.1}
    \frac{g''}{g} = \left( \frac{g'}{g}\right)' + \left(\frac{g'}{g} \right)^2.
\end{equation}
When applied to the second equation in system \eqref{2.1}, a Harnack inequality can be obtained. In order to prepare for that argument, we first prove the following auxiliary estimate.
\begin{lemma}\label{l3.1}
    Assume that $u_0$ and $f$ satisfy \eqref{h1} and \eqref{h2}. Then, there exists $C(t)>0$ such that  
    $$
    \int_\Omega \frac{f(\cdot,t)}{v_\varepsilon(\cdot,t)} + \int_\Omega \frac{|v_{\varepsilon x}(\cdot,t)|^2}{v_\varepsilon^2(\cdot,t) } \leq C(t),
    $$
    for all $\varepsilon \in (0,1)$, $t \in (0,T_{\max, \varepsilon})$.    
\end{lemma}
    \begin{proof}
    Testing the second equation in system \eqref{2.1} by $1/v_\varepsilon$ and integrating by parts, one obtains
    $$
    - \int_\Omega \frac{|v_{\varepsilon x}|^2}{v_\varepsilon^2} + \int_\Omega u_\varepsilon = \int_\Omega \frac{f}{v_\varepsilon}, \quad \text{for all } t \in (0,T_{\max, \varepsilon}).
    $$
    Thus, the boundedness of $u_\varepsilon$ in $L^1(\Omega)$ provided by Lemma \ref{l2.2}, directly provides $C(t) := |\Omega| \cdot \big(1 + ||f||_{L^\infty(\Omega \times (0,t))} \cdot t \hspace{0.05 cm} \big) + \int_\Omega u_0>0$ such that
    $$
    \int_\Omega \frac{f}{v_\varepsilon} + \int_\Omega \frac{|v_{\varepsilon x}|^2}{v_\varepsilon^2} = \int_\Omega u_\varepsilon \leq C(t) \quad \text{for all } t \in (0,T_{\max, \varepsilon}),
    $$
    finishing the proof.
    \end{proof}

\begin{lemma}\label{l3.2}
    Let \eqref{h1} and \eqref{h2} hold. Then, there exists $C(t)>0$ such that
    $$
    \sup_{x \in \Omega} v_{\varepsilon} (x,t) \leq C(t) \inf_{x \in \Omega} v_\varepsilon(x,t),  
    $$
    for all $\varepsilon \in (0,1)$, $t \in (0, T_{\max, \varepsilon})$.
\end{lemma}
\begin{proof}
We divide again the second equation in system \eqref{2.1} again by $v_\varepsilon$. Integrating in space from $0$ to any $x \in \Omega = (0,L)$, and making use of identity \eqref{3.1} yields
$$
\int_0^x \left( \frac{v_{\varepsilon x}}{v_\varepsilon}\right)_x + \int_0^x \left(\frac{v_{\varepsilon x}}{v_\varepsilon} \right)^2   = \int_0^x \frac{v_{\varepsilon xx}}{v_\varepsilon} = \int_0^x u_\varepsilon - \int_0^x \frac{f}{v_\varepsilon}, \quad \text{for all } t \in (0,T_{\max, \varepsilon}).
$$
Due to the Neumann homogeneous condition for $v_\varepsilon$, the nonnegativity of the second term on the left-hand side and the results from Lemma \ref{l2.2} and Lemma \ref{l3.1}, it follows directly that for all $t\in (0,T_{\max, \varepsilon})$, there exists $c_1(t)>0$ independent of $\varepsilon$ such that
$$
\frac{v_{\varepsilon x}(x,t)}{v_\varepsilon(x,t)} \leq \int_\Omega u_\varepsilon + \int_\Omega \frac{f}{v_\varepsilon} \leq c_1(t), \quad \text{for all } x \in \Omega.
$$
Conversely, if the integration is made from $x$ to $L$, the same estimate is obtained for $\displaystyle -\frac{v_{\varepsilon x}(x,t)}{v_\varepsilon(x,t)}$, and thus 
\begin{equation}\label{3.2}
    \left | \frac{v_{\varepsilon x}(x,t)}{v_\varepsilon(x,t)}\right| \leq c_1(t), \quad \text{for all } x\in \Omega, ~t \in (0, T_{\max,\varepsilon}).
\end{equation}
We next note that $\displaystyle \frac{v_{\varepsilon x}(x,t)}{v_\varepsilon(x,t)} = \big(\ln v_\varepsilon (x,t) \big)_x$. Hence, for any $x,y \in \Omega$, integrating \eqref{3.2} results in
\begin{equation} \label{3.3}
v_\varepsilon(x,t) \leq C(t) v_\varepsilon(y,t), \quad \text{for all } t \in (0, T_{\max,\varepsilon}).
\end{equation}
where $C(t) :=e^{|\Omega| \cdot c_1(t)} >0$. Taking infimum and supremum in \eqref{3.3} completes the proof.
\end{proof}
Having established this Harnack inequality, we can next obtain pointwise bounds for $v_\varepsilon(x,t)$ by means of an upper bound for $\inf_{x\in\Omega} v_\varepsilon (x,t)$ and a lower bound for $\sup_{x\in\Omega} v_\varepsilon (x,t)$.

\begin{lemma}\label{l3.3} Assume \eqref{h1} and \eqref{h2}. Then, there for all $\varepsilon\in (0,1)$ the following properties are satisfied for all $t \in (0, T_{\max, \varepsilon})$
$$
\inf_{x\in \Omega} v_\varepsilon(x,t) \leq \frac{\displaystyle \int_\Omega f(\cdot,t)}{\displaystyle \int_\Omega u_0}  \quad \text{and } \quad\sup_{x\in \Omega} v_\varepsilon(x,t) \geq \frac{\displaystyle \int_\Omega f(\cdot,t)}{ ||f||_{L^\infty(\Omega \times (0,t))}  \cdot t + \displaystyle \int_\Omega u_0}~.
$$
Moreover, there exist $c_1(t), c_2(t)>0$ independent of $\varepsilon$ such that
$$
c_1(t) \leq v_\varepsilon(x,t) \leq c_2(t), \quad \text{for all } x \in \Omega, ~t \in (0,T_{\max, \varepsilon}).
$$
In particular, for any $T \in (0,T_{\max, \varepsilon})$ we have
$$
||v_\varepsilon(\cdot, t)||_{L^\infty(\Omega)} \leq c_3(T), \quad \text{for all } t \in (0,T),
$$
for some $c_3(T)>0$ independent of $\varepsilon$.
\end{lemma}
\begin{proof}
    We begin by determining the estimate for $\inf_{x \in \Omega} v_\varepsilon(x,t)$ by contradiction. The proof is based once again on the fact that integrating the second equation in system \eqref{2.1}, one obtains
    \begin{equation}\label{3.4}
        \int_\Omega f = \int_\Omega u_\varepsilon v_\varepsilon, \quad \text{for all } t \in (0,T_{\max, \varepsilon}).
    \end{equation}   
    Assuming no upper bound can be found, then for any $M>0$, there exists $t_M \in (0,T_{\max, \varepsilon})$ such that
    \begin{equation}\label{3.5}
            \inf_{x\in \Omega} v_\varepsilon(x,t_M) > M,
    \end{equation}
    which in particular implies that $v_\varepsilon(x,t_M) > M$ for all $x \in \Omega$. Relying in \eqref{3.4}, we obtain that
    \begin{equation}\label{3.6}
        \int_\Omega f(\cdot, t_M)  > M \int_\Omega u_\varepsilon(\cdot, t_M) \geq M \int_\Omega u_0,
    \end{equation}
    due to the the lower bound for $ \int_\Omega u_\varepsilon$ from Lemma \ref{l2.2}. It directly follows from \eqref{3.6} that then
    \begin{equation} \label{3.7}
    M \leq \frac{\displaystyle \int_\Omega f(\cdot, t_M)}{\displaystyle \int_\Omega u_0}.
    \end{equation}
    Hence \eqref{3.5} cannot be satisfied by any value of $M>0$, only by those satisfying \eqref{3.7}, yielding a contradiction. Thus, the desired bound is obtained.
    \\\\
    The corresponding proof for the upper bound of $\sup_{x\in\Omega} v_\varepsilon(x,t)$ is analogous. If for any small enough $m>0$ one can find $t_m \in (0,T_{\max,\varepsilon})$ such that
    $$
    \sup_{x\in\Omega} v_\varepsilon(x,t_m) < m,
    $$
    then using \eqref{3.4} again provides
    \begin{equation}\label{3.8}
    \int_\Omega f(\cdot, t_m) < m \int_\Omega u_\varepsilon(\cdot, t_m) \leq m \left(|\Omega| \cdot \big(1 + ||f||_{L^\infty(\Omega \times (0,t_m))} \cdot t_m  \big) + \int_\Omega u_0 \right),
    \end{equation}
    again by Lemma \ref{l2.2}. This time \eqref{3.8} implies a largeness condition for $m$, and thus any small enough $m>0$ results in a contradiction, leading to the upper bound for $\sup_{x\in\Omega} v_\varepsilon(x,t)$.
    \\\\
    With these estimates, the pointwise bounds for $v_\varepsilon(x,t)$ are a direct consequence of the Harnack inequality proved in Lemma \eqref{l3.2}. In particular, there exists $C(t)>0$ such that
    \begin{equation}\label{3.9}
    \begin{split}
        \frac{\displaystyle \int_\Omega f(\cdot,t)}{ ||f||_{L^\infty(\Omega \times (0,t))}  \cdot t + \displaystyle \int_\Omega u_0} \leq \sup_{x\in\Omega} v_\varepsilon(x,t) \leq C(t) \cdot \inf_{x\in \Omega}v_\varepsilon(x,t) 
        \\ \leq C(t) \cdot \frac{\displaystyle \int_\Omega f(\cdot,t)}{\displaystyle \int_\Omega u_0}, \quad \text{for all } t \in (0,T_{\max, \varepsilon}).
    \end{split}
    \end{equation}
    This yields
        \begin{equation}\label{3.10}
        c_1(t) := \frac{\displaystyle \int_\Omega f(\cdot,t)}{C(t) \left( ||f||_{L^\infty(\Omega \times (0,t))}  \cdot t + \displaystyle \int_\Omega u_0 \right)}  \leq v_\varepsilon(x,t) 
        \leq  C(t) \cdot \frac{\displaystyle \int_\Omega f(\cdot,t)}{\displaystyle \int_\Omega u_0} =: c_2(t),
    \end{equation}
    for all $x \in \Omega$, $t \in (0,T_{\max, \varepsilon})$. Lastly, for fixed $T \in (0,T_{\max,\varepsilon})$, the $L^\infty$ bound for $v_\varepsilon$ follows immediately from the upper bound provided by $c_2$.
\end{proof}
Combining the $L^\infty$ bound for $v_\varepsilon$ with the $L^p$ estimate for $u_\varepsilon$ obtained in Lemma \ref{l2.3}, we finish the section by deriving higher elliptic regularity for $v_\varepsilon$.
\begin{lemma}\label{l3.4}
    Let hypothesis \eqref{h1} and \eqref{h2} hold. Then, for all $p>1$, $\varepsilon \in (0,1)$ and$T \in (0,T_{\max,\varepsilon})$ there exists $C(p,T)>0$ independent of $\varepsilon$ such that
    $$
    ||v_\varepsilon(\cdot, t)||_{W^{2,p}(\Omega)} \leq C(p,T),
    $$
    and moreover, for some $\alpha \in (0,1)$
    $$
    ||v_\varepsilon(\cdot, t)||_{C^{1,\alpha}(\Omega)} + ||v_{\varepsilon x}(\cdot,t)||_{L^\infty(\Omega)} \leq C(p,T).
    $$
    for all $t \in (0,T)$.
\end{lemma}
\begin{proof}
    We rewrite the second equation in system \eqref{2.1} as
    $$
    - v_{\varepsilon xx} + v_\varepsilon = f(x,t) + v_\varepsilon - u_\varepsilon v_\varepsilon.
    $$
    Given $T \in (0,T_{\max,\varepsilon})$, by assumption \eqref{h2} and the results from Lemma \ref{l2.3} and Lemma \ref{l3.3}, the right-hand side satisfies
    $$
    ||f(\cdot,t) + v_\varepsilon(\cdot,t) - u_\varepsilon(\cdot,t) v_\varepsilon(\cdot,t)||_{L^p(\Omega)} \leq c_1(p,T), \quad \text{for all } t \in (0,T),
    $$
    for all $p> 1$. Then, classical theory for linear elliptic equations becomes applicable (see \cite{gilbard-trudinger} for instance), providing for any $p>1$ a $C(p,T)>0$ such that
    $$
    ||v_\varepsilon(\cdot, t)||_{W^{2,p}(\Omega)} \leq C(p,T), \quad \text{for all } t \in (0,T_{\max, \varepsilon}).
    $$
    Next, in our one-dimensional setting for any $p>1$, the embedding $W^{2,p}(\Omega) \hookrightarrow C^{1,\alpha}(\Omega)$, with $\alpha = 1- \frac{1}{p} \in (0,1)$ is continuous, yielding the uniform estimate for $||v_\varepsilon(\cdot, t)||_{C^{1,\alpha}(\Omega)}$. We note that $||v_{\varepsilon x}(\cdot,t)||_{L^\infty(\Omega)}$ is controlled by the $C^{1,\alpha}-$norm, which completes the proof.
\end{proof}

\section{Time regularity}\label{s4}

With the previous sections being devoted to spatial regularity, we now turn to obtain uniform time regularity properties for $u_\varepsilon$ and $v_\varepsilon$. As a first step, making use of the pointwise lower bound for $v_\varepsilon$ proved in Lemma \ref{l3.3}, we can obtain the following estimate.
\begin{lemma}\label{l4.1}
    Let \eqref{h1} and \eqref{h2} hold. Then, for all $\varepsilon \in (0,1)$, $p >1$, and all $T\in (0,T_{\max,\varepsilon})$, there exists $C(p,T)>0$ independent of $\varepsilon$ such that
    $$
    \int_0^T \int_\Omega \left|\left(u_\varepsilon^{\frac{p+1}{2}}\right)_x \hspace{0.01cm}\right|^2 \leq C(p,T).
    $$
\end{lemma}
\begin{proof}
    Given $T\in (0,T_{\max, \varepsilon})$, we only need to note that, by Lemma \ref{l3.3}, there exists $c_1(T)>0$ satisfying $v_\varepsilon(x,t) \geq c_1(T)$ for all $x \in \Omega$, $t \in (0,T)$.
    \\\\
    Thereby, taking any $p> 1$, the estimate from Lemma \ref{l2.3} provides $c_2(p,T)>0$ such that
    \begin{equation}\label{4.1}
        c_1(T) \int_0^T \int_\Omega u_\varepsilon^{p-1} |u_{\varepsilon x}|^2 \leq \int_0^T \int_\Omega u_\varepsilon^{p-1} v_\varepsilon |u_{\varepsilon x}|^2 \leq c_2(p,T).
    \end{equation}
    Thus, as $\big|\big(u_\varepsilon^{\frac{p+1}{2}}\big)_x \hspace{0.01cm}\big|^2 = \left(\frac{p+1}{2}\right) u_\varepsilon^{p-1} |u_{\varepsilon x}|^2$, the proof is complete by taking $C(p,T) := \displaystyle \frac{(p+1) c_2(p,T)}{2c_1(p,T)}>0.$
\end{proof}
Next, we can obtain a the bound for $(u_{\varepsilon t})_{\varepsilon \in (0,1)}$ in $L^2\big((0,T); (W^{1,2}(\Omega))^*\big)$, which will later prove to be important in obtaining the desired time regularity for $v_\varepsilon$.
\begin{lemma}\label{l4.2}
    Assume \eqref{h1} and \eqref{h2}. Then, for all $\varepsilon \in (0,1)$ and all $T \in (0,T_{\max, \varepsilon})$, one can find $C(T)>0$ independent of $\varepsilon$ satisfying
    $$
    ||u_{\varepsilon t}||_{L^2\big((0,T); (W^{1,2}(\Omega))^*\big)} \leq C(T).
    $$
\end{lemma}
\begin{proof}
    Taking any $\varepsilon \in (0,1)$ and $T \in (0,T_{\max,\varepsilon})$, we look for a bound for
    $$
    ||u_{\varepsilon t}||_{L^2\big((0,T); (W^{1,2}(\Omega))^*\big)}^2 = \int_0^T \left(||u_{\varepsilon t}||_{\big(W^{1,2}(\Omega)\big)^*} \right)^2 dt,
    $$
    independent of $\varepsilon$. To bound the $\big(W^{1,2}(\Omega)\big)^*-$norm, we consider an arbitrary $\psi \in W^{1,2}(\Omega)$ with $||\psi||_{W^{1,2}(\Omega)} \leq 1$, and estimate the dual pairing integrating by parts as follows.
    \begin{equation}\label{4.2}
    \begin{split}
        \left | \int_\Omega u_{\varepsilon t} \psi \right| &= \left | \int_\Omega \Big((u_\varepsilon v_\varepsilon u_{\varepsilon x})_x -(u_\varepsilon^2 v_\varepsilon v_{\varepsilon x})_x + u_\varepsilon v_\varepsilon \Big) \psi \right| \\[1.5 ex]
       & \leq \left|\int_\Omega u_\varepsilon v_\varepsilon u_{\varepsilon x} \psi_x \right| + \left|\int_\Omega u_\varepsilon^2 v_\varepsilon v_{\varepsilon x} \psi_x \right| + \left|\int_\Omega u_\varepsilon v_\varepsilon \psi \right|, \quad \text{for all } t \in (0,T).
    \end{split}
    \end{equation}
    The last two terms can be easily estimated thanks to the $L^p(\Omega)$ bounds derived for $u_\varepsilon$ in Lemma \ref{l2.3} and the boundedness in $L^\infty(\Omega)$ of $v_\varepsilon$ and $v_{\varepsilon x}$. First, using Young's inequality, we have
    \begin{equation}\label{4.3}
    \begin{split}        
        \left|\int_\Omega u_\varepsilon^2 v_\varepsilon v_{\varepsilon x} \psi_x \right| &\leq \frac{||v_\varepsilon||_{L^\infty(\Omega)} \cdot ||v_{\varepsilon x}||_{L^\infty(\Omega)}}{2} \left( \int_\Omega u_\varepsilon^4 + \int_\Omega |\psi_x|^2  \right) \\[1.5 ex]
        &\leq c_1(T), \quad \text{for all } t \in (0,T).
    \end{split}
    \end{equation}
    Similarly, 
    \begin{equation}\label{4.4}
        \left|\int_\Omega u_\varepsilon v_\varepsilon \psi \right| \leq \frac{||v_\varepsilon||_{L^\infty(\Omega)}}{2} \left( \int_\Omega u_\varepsilon^2 + \int_\Omega |\psi|^2\right) \leq c_2(T), \quad \text{for all } t \in (0,T).
    \end{equation}
    Thus, substituting \eqref{4.3} and \eqref{4.4} into \eqref{4.2} one obtains
    \begin{equation}\label{4.5}
    \begin{split}
        \int_0^T \left| \int_\Omega u_{\varepsilon t} \psi \right|^2 &\leq \int_0^T \left(\left|\int_\Omega u_\varepsilon v_\varepsilon u_{\varepsilon x} \psi_x \right| + \big(c_1(T)+c_2(T)\big) \right)^2 \\[1.5 ex]
        &\leq c_3(T) + 2 \int_0^T \left(\int_\Omega u_\varepsilon v_\varepsilon| u_{\varepsilon x}| |\psi_x| \right)^2,
    \end{split}
    \end{equation}
    where $c_3(T):=2T  \big(c_1(T)+c_2(T)\big)^2>0$. 
    \\\\
    The final step consists of using the estimate obtained in Lemma \ref{l4.1} for $p = 3$ to bound the last term in \eqref{4.5}. To do so, the Cauchy-Schwarz inequality yields
    \begin{equation*}
        \begin{split}
             &\int_0^T\left(\int_\Omega u_\varepsilon v_\varepsilon| u_{\varepsilon x}| |\psi_x| \right)^2  \leq \int_0^T\left(\int_\Omega u_\varepsilon^2 v_\varepsilon^2 |u_{\varepsilon x}|^2 \right) \cdot\left(\int_\Omega |\psi_x|^2 \right) \\[1.5 ex]
            & \hspace{1.5 cm}  \leq \int_0^T \left(||v_\varepsilon||^2_{L^\infty(\Omega)} \int_\Omega u_\varepsilon^2 |u_{\varepsilon x}|^2 \right) \leq c_4(T) \int_0^T \int_\Omega \left|\big(u_\varepsilon^2\big)_x\right|^2 \leq C(T).
        \end{split}
    \end{equation*}
    Since the above estimate holds for any $\psi \in W^{1,2}(\Omega)$ with $||\psi||_{W^{1,2}(\Omega)} \leq 1$, we can conclude that $||u_{\varepsilon t}||_{L^2\big((0,T); (W^{1,2}(\Omega))^*\big)} \leq C(T),$ and thus the proof is complete.
\end{proof}
We next turn to the study of time regularity properties for $v_\varepsilon$. The main difference with respect to $u_\varepsilon$ lies on the apparent lack of information regarding $v_{\varepsilon t}$, as the second equation in \eqref{2.1} is elliptic. To overcome this, given $h>0$, we define
\begin{equation}\label{4.6}
    \delta_h v_\varepsilon (x,t) := v_\varepsilon(x,t+h) - v_\varepsilon(x,t), \quad \text{for all } x\in \Omega, ~t \in (0,T_{\max,\varepsilon} - h).
\end{equation}
By taking differences in the second equation in system \eqref{2.1}, one obtains
    \begin{equation*}
    \begin{split}
         - v_{\varepsilon xx}(x,t+h) + v_{\varepsilon xx}(x,t) = - u_\varepsilon(x,t+h)v_\varepsilon(x,t+h) + u_\varepsilon(x,t) v_{\varepsilon}(x,t)
         + f(x,t+h)-f(x,t),
    \end{split}
    \end{equation*}
for all $x \in \Omega$, $t \in (0,T-h)$. Rewriting the above expression, and similarly defining $\delta_h u_\varepsilon(x,t)$ and $\delta_h f(x,t)$ following \eqref{4.6}, we have that for any $t \in (0,T-h)$, $\delta_h v_\varepsilon$ satisfies the following linear elliptic equation
\begin{equation}\label{4.7}
   \begin{cases}
      -\big( \delta_h v_\varepsilon(x,t) \big)_{xx} + u_\varepsilon(x,t+h)\hspace{0.1 cm} \delta_h v_\varepsilon(x,t) = - v_\varepsilon(x,t) \delta_h u_\varepsilon(x,t) + \delta_h f(x,t), & x \in \Omega, \\[1.5 ex]
       \hspace{0.3 cm}\big(\delta_h v_\varepsilon(x,t) \big)_x = 0,  & x \in \partial \Omega,
   \end{cases}
\end{equation}
where the Neumann homogeneous boundary conditions for $v_\varepsilon$ are naturally inherited by $\delta_h v_\varepsilon$. We note that as a result of hypothesis \eqref{h2}, for any $T \in (0,T_{\max,\varepsilon})$ we have
\begin{equation}\label{f-aux}
     \int_\Omega \Big|\delta_h f(\cdot,t) \Big| \leq |\Omega| \cdot ||f_t||_{L^\infty(\Omega \times (0,T))} \cdot h, \quad \text{for all } t \in (0,T-h)
\end{equation}
which, in combination with the bound for $u_{\varepsilon t}$ in $L^2\big((0,T); (W^{1,2}(\Omega))^*\big)$ from Lemma \ref{l4.2} allow us to prove the following results.
\begin{lemma}\label{l4.3}
    Assume \eqref{h1} and \eqref{h2}. Then, for all $h>0$, and all $\varepsilon\in (0,1)$, given $T \in (0,T_{\max, \varepsilon})$, there exists $C(T)>0$ independent of $h$ and $\varepsilon$ such that $\delta_h v_\varepsilon$ as defined in \eqref{4.6} satisfies
    $$
    \Big|\Big|\big(\delta_h v_\varepsilon(\cdot, t)\big)_x\Big| \Big|_{L^2(\Omega)}^2 \leq C(T) \left( h^{1/2} + h \right), \quad \text{for all } t \in (0,T - h).
    $$
\end{lemma}
\begin{proof}
    We start by noting that for all $\varepsilon \in (0,1)$, given $T \in (0,T_{\max, \varepsilon})$ due to the $L^\infty$ bound for $v_\varepsilon$ from Lemma \ref{l3.3}, there exists $c_1(T)>0$ independent of $\varepsilon$ such that $||\delta_h v_\varepsilon||_{L^\infty(\Omega)}<c_1(T)$ for all $t \in (0,T-h)$.
    \\\\
    Next, testing equation \eqref{4.7} by $\delta_h v_\varepsilon(x,t)$ we get
    \begin{equation*}
    \begin{split}
    \int_\Omega &\Big|\big(\delta_h v_\varepsilon(\cdot,t) \big)_x \Big|^2 + \int_\Omega u_\varepsilon(\cdot,t+h)\big(\delta_h v_\varepsilon(\cdot,t)\big)^2 
    \\[1.5 ex]
    &  = - \int_\Omega  \delta_h u_\varepsilon(\cdot,t) \delta_h v_\varepsilon(\cdot,t) v_\varepsilon(\cdot,t)+ \int_\Omega \delta_h f(\cdot,t) \delta_h v_\varepsilon(\cdot,t) \quad \text{for all } t \in (0,T-h).
    \end{split}        
    \end{equation*}
    As the second term on the left-hand side is nonnegative, we obtain
    \begin{equation}\label{4.8}
        \begin{split}
            \int_\Omega \Big|\big(\delta_h v_\varepsilon(\cdot,t) \big)_x \Big|^2 & \leq \left| \int_\Omega  \delta_h u_\varepsilon(\cdot,t) \delta_h v_\varepsilon(\cdot,t)v_\varepsilon(\cdot,t)\right| + \left|\int_\Omega \delta_h f(\cdot,t) \delta_h v_\varepsilon(\cdot,t) \right| \\[1.5ex]
           & \leq \left| \int_\Omega \delta_h u_\varepsilon(\cdot,t) \delta_h v_\varepsilon(\cdot,t)  v_\varepsilon(\cdot,t) \right| + ||\delta_h v_\varepsilon(\cdot,t)||_{L^\infty(\Omega)} \int_\Omega \Big|\delta_h f(\cdot,t)\Big|\\[1.5ex]
           & \leq \left| \int_\Omega  \delta_h u_\varepsilon(\cdot,t) \delta_h v_\varepsilon(\cdot,t) v_\varepsilon(\cdot,t)\right| + c_1(T) \int_\Omega \Big|\delta_h f(\cdot,t)\Big|,
        \end{split}
    \end{equation}
    for all $t \in (0,T-h)$. The remaining of the proof consists of estimating the first term on the right-hand side of \eqref{4.8}.
    \\\\
    As by Lemma \ref{l3.3}, $v_\varepsilon$ is bounded in $W^{2,p}(\Omega)$ for all $p>1$ independently of $\varepsilon$, in particular for $p=2$, we obtain that $v_\varepsilon$ is uniformly bounded in $W^{1,2}(\Omega)$ by the continuity of the embedding $W^{2,2}(\Omega) \hookrightarrow W^{1,2}(\Omega)$. It follows directly that $\delta_h v_\varepsilon(\cdot,t)$ is also uniformly bounded in $W^{1,2}(\Omega)$, and a result of their $L^\infty(\Omega)$ estimates, we have that 
    $$||\delta_h v_\varepsilon(\cdot,t) v_\varepsilon(\cdot,t)||_{W^{1,2}(\Omega)}\leq c_2(T), \quad \text{for all } t \in (0,T-h),$$
    for some $c_2(T)>0$. Therefore, if $\delta_h u_\varepsilon \in (W^{1,2}(\Omega))^*$ for all $t \in (0,T-h)$, we can interpret $ \left| \int_\Omega  \delta_h u_\varepsilon(\cdot,t) \delta_h v_\varepsilon(\cdot,t) v_\varepsilon(\cdot,t)\right|$ as the dual pairing, leading to
    \begin{equation}\label{4.9}
    \begin{split}
         &\left| \int_\Omega  \delta_h u_\varepsilon(\cdot,t) \delta_h v_\varepsilon(\cdot,t) v_\varepsilon(\cdot,t)\right| = \Big|\big \langle  \delta_h u_\varepsilon(\cdot,t),\hspace{0.05 cm}\delta_h v_\varepsilon(\cdot,t)v_\varepsilon(\cdot,t)  \big\rangle \Big| \\[1.5 ex]
         &\hspace{1 cm}\leq ||\delta_h v_\varepsilon(\cdot,t) v_\varepsilon(\cdot,t)||_{W^{1,2}(\Omega)} \cdot ||\delta_h u_\varepsilon(\cdot,t)||_{(W^{1,2}(\Omega))^*} \leq c_2(T)\cdot||\delta_h u_\varepsilon(\cdot,t)||_{(W^{1,2}(\Omega))^*},
    \end{split}
    \end{equation}
    for all $t \in (0,T-h)$. Thus, we need to estimate the $(W^{1,2}(\Omega))^*-$norm of $\delta_h u_\varepsilon$, for which we make use of the bound for $u_{\varepsilon t}$ in $L^2\big((0,T);(W^{1,2}(\Omega))^*\big)$ proved in Lemma \ref{l4.2}. In particular, we have
    \begin{equation}\label{4.10}
        \delta_h u_\varepsilon(\cdot,t) = u_\varepsilon(\cdot,t+h) - u_\varepsilon(\cdot,t) = \int_t^{t+h} u_{\varepsilon t}(\cdot,s) ~ds \quad \text{in } (W^{1,2}(\Omega))^*.
    \end{equation}
    Consequently, using the Cauchy-Schwarz inequality
    \begin{equation}\label{4.11}
        \begin{split}
            ||\delta_h u_\varepsilon(\cdot, t)||&_{(W^{1,2}(\Omega))^*} = \sup_{\||\psi||_{W^{1,2}(\Omega)}\leq1} \left| \int_\Omega \delta_h u_\varepsilon(\cdot,t) \psi \right| = \sup_{\||\psi||_{W^{1,2}(\Omega)}\leq1} \left| \int_\Omega \left( \int_t^{t+h} u_{\varepsilon t}(\cdot,s) ~ds\right) \psi \right| 
            \\[1.5 ex] &\leq \sup_{\||\psi||_{W^{1,2}(\Omega)}\leq1} \int_t^{t+h}  \left|\int_\Omega u_{\varepsilon t}(\cdot,s) \psi \right|  = \int_t^{t+h} ||u_{\varepsilon t}||_{(W^{1,2}(\Omega))^*} \\[1.5 ex]
            &\leq \left(\int_t^{t+h} 1\right)^{1/2} \cdot \left( \int_t^{t+h} ||u_{\varepsilon t}||^2_{(W^{1,2}(\Omega))^*}\right)^{1/2} = h^{1/2} \cdot ||u_{\varepsilon t}||_{L^2\big((0,T); (W^{1,2}(\Omega))^*\big)} \\[1.5 ex] &\leq c_3(T) \hspace{0.1 cm}h^{1/2}, \quad \text{for all } t \in (0,T-h),
        \end{split}
    \end{equation}
    for some $c_3(T)>0$. Directly substituting \eqref{4.11} into \eqref{4.9} and subsequently into \eqref{4.8}, we arrive at
    \begin{equation*}
        \begin{split}
            \int_\Omega \Big|\big(\delta_h v_\varepsilon(\cdot,t) \big)_x \Big|^2 \leq c_2(T)\hspace{0.05 cm} c_3(T)\hspace{0.05 cm} h^{1/2} + c_1(T)  \int_\Omega \Big| \delta_h f(\cdot,t)\Big|, \quad \text{for all } t\in(0,T-h).
        \end{split}
    \end{equation*}
    and \eqref{f-aux} finishes the proof, with $C(T) := \max\{c_2(T) \hspace{0.05cm}c_3(T), c_1(T) \cdot |\Omega| \cdot ||f_t||_{L^\infty(\Omega \times (0,T))}\} >0$.
\end{proof}
Next, a similar bound can be proven for $\delta_h v_\varepsilon$ through the Poincaré inequality.
\begin{lemma}\label{l4.4}
    Let \eqref{h1} and \eqref{h2} hold. Then, for all $h>0$, and all $\varepsilon\in (0,1)$, given $T \in (0,T_{\max, \varepsilon})$, there exists $C(T)>0$ independent of $h$ and $\varepsilon$ such that $\delta_h v_\varepsilon$ as defined in \eqref{4.6} satisfies
    $$
    \Big|\Big|\delta_h v_\varepsilon(\cdot, t)\Big| \Big|_{L^2(\Omega)} \leq C(T) \left(h^{1/4} + h^{1/2} + h\right), \quad \text{for all } t \in (0,T - h).
    $$
\end{lemma}
\begin{proof}
    Given $\varepsilon \in (0,1)$, for any $T<T_{\max,\varepsilon}$, by considering 
    $$
    \overline{\delta_h v_\varepsilon}(t) := \frac{1}{|\Omega|} \int_\Omega \delta_h v_\varepsilon(\cdot,t), \quad t \in (0,T),
    $$
    the Poincaré inequality provides $c_1>0$, depending only on $\Omega$ such that
    \begin{equation}\label{4.12}
        \Big|\Big|\delta_h v_\varepsilon(\cdot,t) - \overline{\delta_h v_\varepsilon}(t)\Big|\Big|_{L^2(\Omega)} \leq c_1 \Big|\Big|\big(\delta_h v_\varepsilon(\cdot,t)\big)_x\Big|\Big|_{L^2(\Omega)}, \quad \text{for all } t \in (0,T).
    \end{equation}
    Thus, as $||(\delta_h v_\varepsilon)_x||_{L^2(\Omega)}$ is controlled by Lemma \ref{l4.3}, we only have to estimate the mean $\overline{\delta_h v_\varepsilon}$.
    \\\\
    To do so, as $\delta_h v_\varepsilon$ satisfies equation \eqref{4.7}, upon integrating over $\Omega$ one obtains
    \begin{equation}\label{4.13}
    \begin{split}
        \int_\Omega u_\varepsilon(\cdot,t+h) &\big(\delta_h v_\varepsilon(\cdot,t) - \overline{\delta_h v_\varepsilon}(t) \big) + \int_\Omega u_\varepsilon(\cdot, t+h) \hspace{0.1 cm}\overline{\delta_h v_\varepsilon}(t) \\[1.5 ex]&= - \int_\Omega v_\varepsilon(\cdot,t) \delta_h u_\varepsilon(\cdot,t) + \int_\Omega \delta_h f(\cdot,t), \quad \text{for all } t \in (0,T-h),
    \end{split}
    \end{equation}
    where again using the Cauchy-Schwarz inequality, we have that
    \begin{equation}\label{4.14}
        \begin{split}
             \left|\int_\Omega u_\varepsilon(\cdot,t+h) \big(\delta_h v_\varepsilon(\cdot,t) - \overline{\delta_h v_\varepsilon}(t) \big) \right| &\leq ||u_\varepsilon(\cdot,t+h)||_{L^2(\Omega)} \cdot   \Big|\Big|\delta_h v_\varepsilon(\cdot,t) - \overline{\delta_h v_\varepsilon}(t)\Big|\Big|_{L^2(\Omega)} \\[1.5 ex]
             &\leq c_2(T) \hspace{0.1 cm}\Big|\Big|\big(\delta_h v_\varepsilon(\cdot,t)\big)_x\Big|\Big|_{L^2(\Omega)}, \quad \text{for all } t \in (0,T-h),
        \end{split}
    \end{equation}
    for some $c_2(T)>0$, due to the $L^2$ bound for $u_\varepsilon$ from Lemma \ref{l2.3} and the Poincaré inequality \eqref{4.12}.
    \\\\
    Hence, as $\overline{\delta_h v_\varepsilon}$ is constant over $\Omega$, from \eqref{4.13} for all $t \in (0,T-h)$ we obtain
    \begin{equation*}
            \left|\hspace{0.05cm} \overline{\delta_h v_\varepsilon}(t) \right| \int_\Omega u_\varepsilon(\cdot,t+h) \leq \left| \int_\Omega  v_\varepsilon(\cdot,t) \delta_h u_\varepsilon(\cdot,t) \right| +\left| \int_\Omega \delta_h f(\cdot,t) \right| + c_2(T)\hspace{0.1 cm}\Big|\Big|\big(\delta_h v_\varepsilon(\cdot,t)\big)_x\Big|\Big|_{L^2(\Omega)}.
    \end{equation*}
    The first term on the right-hand side can be handled in the same way as done in \eqref{4.9}, since $v_\varepsilon$ is uniformly bounded in $W^{1,2}(\Omega)$ for all $t \in (0,T-h)$. Moreover, Lemma \ref{l4.3} provides control over $||(\delta_h v_\varepsilon)_x||_{L^2(\Omega)}$, and \eqref{f-aux} allows us to estimate the term involving $\delta_h f$. Consequently, we arrive at
    \begin{equation*}
            \left|\hspace{0.05cm} \overline{\delta_h v_\varepsilon}(t) \right| \int_\Omega u_\varepsilon(\cdot,t+h) \leq c_3(T) \left(h^{1/4} + h^{1/2} +  h\right), \quad \text{for all } t \in (0,T-h).
    \end{equation*}
    for some $c_3(T)>0$. Lastly, from Lemma \ref{l2.2}, we know that $\int_\Omega u_\varepsilon(\cdot,t+h) \geq \int_\Omega u_0(x)>0$, and therefore
    \begin{equation}\label{4.15}
        \left|\hspace{0.05cm}  \overline{\delta_h v_\varepsilon}(t) \right| \leq \frac{c_3(T)}{\displaystyle\int_\Omega u_0(x)} \left(h^{1/4} + h^{1/2} + h\right), \quad \text{for all } t \in (0,T-h).
    \end{equation}
    Having estimated $\overline{\delta_h v_\varepsilon}(t)$, we can conclude the proof, as we have that
    \begin{equation*}
    \begin{split}
        \Big|\Big|\delta_h v_\varepsilon(\cdot,t)\Big|\Big|_{L^2(\Omega)} &\leq \Big|\Big|\delta_h v_\varepsilon(\cdot,t) - \overline{\delta_h v_\varepsilon}(t)\Big|\Big|_{L^2(\Omega)} + \big|\big|\overline{\delta_h v_\varepsilon}(t)  \big| \big|_{L^2(\Omega)} \\[1.5ex]
        &\leq c_1(T) \Big|\Big|\big(\delta_h v_\varepsilon(\cdot,t)\big)_x\Big|\Big|_{L^2(\Omega)} + |\Omega|^{1/2} \left| \hspace{0.05cm} \overline{\delta_h v_\varepsilon}(t) \right|
        \\[1.5 ex] &\leq C(T) \left(h^{1/4} + h^{1/2} + h\right), \quad \text{for all } t \in (0,T-h),
    \end{split}
    \end{equation*}
    for some $C(T)>0$.
\end{proof}
Lastly, a combination of these two Lemmas results in the equicontinuity in time of $\{v_\varepsilon\}_{\varepsilon \in (0,1)}$.
\begin{lemma}\label{l4.5}
    Assume $u_0$ and $f$ satisfy \eqref{h1} and \eqref{h2}. Then, for any $\varepsilon \in (0,1)$ and $T < T_{\max, \varepsilon}$, there exists $\alpha>0$ and $C(T)>0$ independent of $\varepsilon$ such that for any $x \in \Omega$, $t_1,t_2 \in (0,T)$ with $|t_1-t_2|< 1$, one has
    $$
    \big|\hspace{0.02 cm}v_\varepsilon(x,t_2) - v_\varepsilon(x,t_1)\big| \leq C(T) \hspace{0.1 cm} \big|\hspace{0.02 cm}t_2 - t_1\big|^\alpha.
    $$
\end{lemma}

\begin{proof}
    For our one-dimensional setting, as $\Omega$ is bounded, the continuity of the embedding $W^{1,2}(\Omega) \hookrightarrow L^\infty(\Omega)$ provides $c_1>0$ such that for any $w \in W^{1,2}(\Omega)$
    $$
    ||w||_{L^\infty(\Omega)} \leq c_1 ||w||_{W^{1,2}(\Omega)}.
    $$
    In particular, for any $x\in \Omega$, taking $h := |t_1-t_2|<1$ and $t := \min\{t_1,t_2\}$, we have
    \begin{equation*}
        \big|\hspace{0.05cm}v_\varepsilon(x,t_2) - v_\varepsilon(x,t_1) \big| \leq \big|\big|\delta_h v_\varepsilon(\cdot,t)\big|\big|_{L^\infty(\Omega)} \leq c_1 \left( \big|\big|\delta_h v_\varepsilon(\cdot,t)\big|\big|_{L^2(\Omega)} + \big|\big|\big(\delta_h(v_\varepsilon(\cdot,t)\big)_x \big|\big|_{L^2(\Omega)}\right).
    \end{equation*}
    Thus, invoking Lemmas \ref{l4.3} and \ref{l4.4}, considering $h=1/4$, we obtain
    \begin{equation}\label{4.17}
        \big|\hspace{0.05cm}v_\varepsilon(x,t_2) - v_\varepsilon(x,t_1) \big| \leq c_2(T) h^\alpha = c_2(T) |t_1-t_2|^\alpha,
    \end{equation}
    finishing the proof.
\end{proof}
Having obtained the desired time regularity for $v_\varepsilon$, one important aspect about $u_\varepsilon$ has to still be taken into account. As unfortunately Lemma \ref{l4.1} only covers the case $p>1$, we do not have any direct control over $u_{\varepsilon x}$. In view of applying the Aubin-Lions lemma to obtain a convergent subsequence, we extend the analysis to the sequence $\big(u_{\varepsilon}^{\frac{p+1}{2}}\big)_{\varepsilon \in (0,1)}$ for $p>1$ to match the result in Lemma \ref{l4.1}. For the time derivative of this new sequence, a similar bound can be proved, although in a less restrictive space.
\begin{lemma}\label{l4-aux}
    Let \eqref{h1} and \eqref{h2} hold and let $p>1$. Then, for all $\varepsilon \in (0,1)$ and all $T \in (0,T_{\max, \varepsilon})$, one can find $C(p,T)>0$ independent of $\varepsilon$ satisfying
    $$
    \Big|\Big|\partial_t\big(u_{\varepsilon}^{\frac{p+1}{2}}\big) \Big|\Big|_{L^1\big((0,T); (W^{3,2}(\Omega))^*\big)} \leq C(p,T).
    $$
\end{lemma}
\begin{proof}
    The proof strategy is very similar to the one used in Lemma \ref{l4.2}, making use of the integrability property in Lemma \ref{l4.1}. This time, taking $\varepsilon \in (0,1)$ and $T \in (0,T_{\max, \varepsilon})$, we bound
    $$
    \Big|\Big| \partial_t\big(u_{\varepsilon}^{\frac{p+1}{2}}\big) \Big|\Big|_{L^1\big((0,T); (W^{3,2}(\Omega))^*\big)} = \int_0^T \Big|\Big|\partial_t \big(u_{\varepsilon}^{\frac{p+1}{2}}\big) \Big|\Big|_{\big(W^{3,2}(\Omega)\big)^*}  dt,
    $$
    independently of $\varepsilon$. Taking an arbitrary $\psi \in W^{3,2}(\Omega)$ with $||\psi||_{W^{3,2}(\Omega)} \leq 1$, we bound the norm in $(W^{3,2}(\Omega))^*$.
    By computing the time derivative, we obtain
    \begin{equation}\label{4a1}
    \begin{split}
      \frac{2}{p+1}  \int_\Omega \big(u_\varepsilon^{\frac{p+1}{2}}\big)_t \hspace{0.1 cm} \psi = \int_\Omega u_\varepsilon^{\frac{p-1}{2}} u_{\varepsilon t} \psi =      
      \frac{p-1}{2}\int_\Omega u_\varepsilon^{\frac{p-1}{2}}v_\varepsilon u_{\varepsilon x}^2 \psi + \int_\Omega u_\varepsilon^{\frac{p+1}{2}} v_\varepsilon u_{\varepsilon x} \psi_x \\[1.5 ex]
      +\frac{p-1}{2}\int_\Omega u_\varepsilon^{\frac{p+1}{2}} v_\varepsilon u_{\varepsilon x} v_{\varepsilon x} \psi + \int_\Omega u_\varepsilon^{\frac{p+3}{2}} v_\varepsilon v_{\varepsilon x} \psi_x + \int_\Omega u_\varepsilon^{\frac{p+1}{2}} v_\varepsilon \psi =: \sum_{i=1}^5 I_i ,
    \end{split}
    \end{equation}
    for all $t \in (0,T)$. To bound the terms $I_i$, $i \in \{1,\dots, 5\}$, we rely on the continuity of embedding $W^{3,2}(\Omega) \hookrightarrow W^{1,\infty}(\Omega)$, providing $c>0$ such that $||\psi||_{L^\infty(\Omega)} + ||\psi_x||_{L^\infty(\Omega)} \leq c$. In this way, for all $t \in (0,T)$ we have
    \begin{equation}\label{4a2}
        |I_1| \leq \frac{p-1}{2} ||\psi||_{L^\infty(\Omega)} ||v_\varepsilon||_{L^\infty(\Omega)} \int_\Omega u_\varepsilon^{\frac{p-1}{2}} |u_{\varepsilon x}|^2 \leq c_1(p,T) \int_\Omega \Big|\big(u_{\varepsilon}^\frac{p+1}{2}\big)_x \Big|^2,
    \end{equation}
    and Young's inequality yields 
    \begin{equation}\label{4a3}
        |I_2| \leq \int_\Omega u_\varepsilon^{\frac{p+1}{2}} v_\varepsilon |u_{\varepsilon x}|\hspace{0.05cm} |\psi_x| \leq \frac{1}{2} \int_\Omega u_\varepsilon^{p+1} |u_{\varepsilon x}|^2 + \frac{1}{2}\int_\Omega v_\varepsilon^2 |\psi_x|^2 \leq c_2(p,T)\left( 1+ \int_\Omega \Big|\big(u_{\varepsilon}^{p+2}\big)_x \Big|^2 \right),   
    \end{equation}
    as well as
    \begin{equation}\label{4a4}
    \begin{split}
        |I_3| \leq \frac{p-1}{2} \int_\Omega u_\varepsilon^{\frac{p+1}{2}} v_\varepsilon |u_{\varepsilon x}| \hspace{0.05 cm} |v_{\varepsilon x}| \hspace{0.05 cm} |\psi| 
        &\leq ||v_\varepsilon||_{L^\infty(\Omega)} ||v_{\varepsilon x}||_{L^\infty(\Omega)} \frac{p-1}{4} \left(\int_\Omega u_\varepsilon^{p+1} |u_{\varepsilon x}|^2 + \int_\Omega |\psi|^2\right)\\
        &\leq c_3(p,T) \left( 1 + \int_\Omega \Big|\big(u_{\varepsilon}^{p+2}\big)_x \Big|^2  \right), \quad \text{for all } t \in (0,T).
    \end{split}
    \end{equation}
    Similarly, making use of the $L^p$ estimates for $u_\varepsilon$ obtained in Lemma \ref{l2.3} we have
    \begin{equation}\label{4a5}
        |I_4| \leq \frac{||v_\varepsilon||_{L^\infty(\Omega)} ||v_{\varepsilon x}||_{L^\infty(\Omega)}}{2} \left(\int_\Omega u_\varepsilon^{p+3} + \int_\Omega |\psi_x|^2\right) \leq c_4(p,T), \quad \text{ for all } t \in (0,T),
    \end{equation}
    and lastly
    \begin{equation}\label{4a6}
        |I_5| \leq \frac{||v_\varepsilon||_{L^\infty(\Omega)}}{2} \left( \int_\Omega u_\varepsilon^{p+1} + \int_\Omega |\psi|^2\right) \leq c_5(p,T), \quad \text{ for all } t \in (0,T).
    \end{equation}
    Therefore, from \eqref{4a2}-\eqref{4a6}, there exists $c_6(p,T)>0$ such that
    $$
   \Big|\Big|\partial_t \big(u_{\varepsilon}^{\frac{p+1}{2}}\big) \Big|\Big|_{\big(W^{3,2}(\Omega)\big)^*}\leq c_6(p,T) \left( 1+ \int_\Omega \Big|\big(u_{\varepsilon}^{p+2}\big)_x \Big|^2 + \int_\Omega \Big|\big(u_{\varepsilon}^\frac{p+1}{2}\big)_x \Big|^2 \right), \quad \text{ for all } t \in (0,T),
    $$
    and a time integration in view of Lemma \ref{l4.1} finishes the proof.
\end{proof}

\section{Global existence of the limit $(u,v)$}\label{s5}

After the previous lemmas ensuring boundedness and regularity properties of the sequences $\big(u_\varepsilon^\frac{p+1}{2}\big)_{\varepsilon \in (0,1)}$ and $(v_\varepsilon)_{\varepsilon \in (0,1)}$, in this final section we show that all regularized solutions in fact exist globally in time, and that a subsequence can be extracted such that it converges to the global weak solution of system \eqref{1.1}.

\begin{lemma}\label{l5.1}
    Assume \eqref{h1} and \eqref{h2}. Then, for all $\varepsilon \in (0,1)$, $T_{\max, \varepsilon} = \infty$.
\end{lemma}
\begin{proof}
    Due to the $L^p$ bounds for $u_\varepsilon$ provided by Lemma \ref{l2.3} and the regularity of $v_\epsilon$ entailed by Lemmas \ref{l3.3} and \ref{l3.4}, a standard argument can be applied to show that $T_{\max,\varepsilon}$ cannot be finite. 
    \\\\
    In particular, for all $\varepsilon \in (0,1)$, we rewrite the first equation in \eqref{2.1} as
    \begin{equation}\label{5.1}
    u_{\varepsilon t} = (D_\varepsilon(x,t,u) u_{\varepsilon x})_x + (f_\varepsilon(x,t))_x + g_\varepsilon(x,t), \quad \text{for all } x \in \Omega, ~t \in (0,T_{\max, \varepsilon}),
    \end{equation}    
    where
    $$D_\varepsilon(x,t) = u_\varepsilon v_\varepsilon, \quad f_\varepsilon(x,t) = -u_\varepsilon^2 v_\varepsilon v_{\varepsilon x}, \quad g_\varepsilon(x,t) = u_\varepsilon v_\varepsilon, \quad (x,t) \in \Omega \times (0,T_{\max,\varepsilon}).$$
    By contradiction, if for any $\varepsilon\in(0,1)$, $T_{\max, \varepsilon}<\infty$, then by \eqref{2.06} $\lim_{t \nearrow T_{\max,\varepsilon}} ||u_\varepsilon||_{L^\infty(\Omega)} = \infty$. We use Lemma A.1 in \cite{TaoWinklerA1} to prove that indeed, $ \sup_{t \in (0,T_{\max,\varepsilon})}||u_\varepsilon||_{L^\infty(\Omega)}<\infty$, and thus $T_{\max,\varepsilon}=\infty$.
    \\\\
    For any fixed $\varepsilon \in (0,1)$, by a comparison argument,
    $u_\varepsilon(x,t) \geq \varepsilon >0$ for all $x \in \Omega, ~t \in (0,T_{\max, \varepsilon})$ and by Lemma \ref{l3.3}, there exists $c_1>0$ such that $v_\varepsilon(x,t) > c_1$, for all $x \in \Omega, ~t \in (0,T_{\max, \varepsilon})$. Consequently, $\inf_{(x,t) \in \Omega \times (0,T_{\max, \varepsilon})} D(x,t)>0$.
    \\\\
    Moreover, using Lemmas \ref{l2.3} and \ref{l3.3}, for any $p>1$, we obtain that
    $$
    \sup_{t \in (0,T_{\max,\varepsilon})} \left \{||u_\varepsilon(\cdot,t)||_{L^p(\Omega)} + ||f_\varepsilon(\cdot,t)||_{L^p(\Omega)} + ||g_\varepsilon(\cdot,t)||_{L^p(\Omega)} \right\} <\infty.
    $$
    To apply Lemma A.1 in \cite{TaoWinklerA1}, we only need to ensure that $D_\varepsilon \in C^1\big(\bar{\Omega} \times [0,T_{\max,\varepsilon}) \times [0,\infty)\big)$. Lemma \ref{l2.1} already provides time differentiability for $u_\varepsilon$, but for $v_\varepsilon$, Lemma \ref{l4.5} only gives time continuity. However, differentiating the second equation in \eqref{2.1}, one obtains
    \begin{equation}\label{5.2}
    -(v_{\varepsilon t})_{xx} + u_\varepsilon v_{\varepsilon t} = - u_{\varepsilon t} v_\varepsilon + f_t, \quad x \in \Omega, ~t\in (0,T_{\max,\varepsilon}).
    \end{equation}
    For any fixed $\varepsilon \in (0,1)$, the operator $A_\varepsilon(t) w := -w_{xx} + u_\varepsilon(t) w$ is invertible due to the strict positivity of $u_\varepsilon$. In addition, $h_\varepsilon := - u_{\varepsilon t} v_\varepsilon + f_t \in C^0\big(\Omega \times (0,T_{\max, \varepsilon})\big)$ with $\sup_{t \in (0,T_{\max, \varepsilon})}||h_\varepsilon(\cdot,t)||_{L^p(\Omega)}<\infty$ for any $p>1$ due to the continuity of $u_{\varepsilon t}$ by Lemma \ref{l2.1}, Lemmas \ref{l3.3} and \ref{l4.5} for $v_\varepsilon$ and assumption \eqref{h2} for $f_t$. Thus, by linear elliptic estimates
    \begin{equation}\label{5.3}
        ||v_{\varepsilon t}(\cdot,t)||_{W^{2,p}(\Omega)} \leq C ||h_\varepsilon(\cdot,t)||_{L^p(\Omega)}, \quad \text{for any } p >1 \text{ and all } t\in (0,T_{\max, \varepsilon}), 
    \end{equation}
    where $C = C(\varepsilon,p)>0.$ To prove the time continuity of $v_{\varepsilon t}$, given $h>0$, we take differences in equation \eqref{5.2} as in Lemma \ref{l4.3}, resulting in
    \begin{equation*}
    \begin{split}
    -\Big(v_{\varepsilon t}&(x,t+h)-v_{\varepsilon t}(x,t) \Big)_{xx} + u_\varepsilon(x,t+h) \Big(v_{\varepsilon t}(x,t+h)-v_{\varepsilon t}(x,t) \Big) \\[1.5ex]
     &= h_\varepsilon(x,t+h) - h_\varepsilon(x,t) + \big(u_{\varepsilon t}(x,t+h)-u_{\varepsilon t}(x,t) \big) v_{\varepsilon t}(x,t), \quad x \in \Omega, t \in (0,T_{\max,\varepsilon} - h).
    \end{split}
    \end{equation*}
    The left-hand side corresponds this time to $A_\varepsilon(t+h)\big(v_{\varepsilon t}(x,t+h)-v_{\varepsilon t}(x,t) \big)$, where again the strict positivity of $u_{\varepsilon}$ results in the invertibility of the operator. By the same previous elliptic estimate, we have
    $$
    \big|\big|v_{\varepsilon t}(\cdot,t+h)-v_{\varepsilon t}(\cdot,t)\big|\big|_{W^{2,p}(\Omega)} \leq C_2 \big|\big| h_\varepsilon(\cdot,t+h) - h_\varepsilon(\cdot,t) + \big(u_{\varepsilon t}(\cdot,t+h)-u_{\varepsilon t}(\cdot,t) \big) v_{\varepsilon t}(\cdot,t)\big|\big|_{L^p(\Omega)},
    $$
    for any $p>1$ and some $C_2 = C_2(\varepsilon, p)>0$. Due to the continuity of $h_\varepsilon$, we have that
    $$
    \big|\big| h_\varepsilon(\cdot,t+h) - h_\varepsilon(\cdot,t)\big|\big|_{L^p(\Omega)} \to 0 \quad \text{as } h \to 0,
    $$
    while for the second term, using Hölder's inequality, taking $q,r>1$ such that $\frac{1}{p} = \frac{1}{q} + \frac{1}{r}$, we obtain
    $$
    \big|\big| \big(u_{\varepsilon t}(\cdot,t+h)-u_{\varepsilon t}(\cdot,t) \big) v_{\varepsilon t}(\cdot,t)\big|\big|_{L^p(\Omega)} \leq \big|\big|  u_{\varepsilon t}(\cdot,t+h)-u_{\varepsilon t}(\cdot,t) \big)\big|\big|_{L^q(\Omega)} \cdot  \big|\big| v_{\varepsilon t}(\cdot,t)\big|\big|_{L^r(\Omega)}.
    $$
    By \eqref{5.3}, $||v_{\varepsilon t}(\cdot,t)||_{L^r(\Omega)}$ is bounded, while the time continuity of $u_{\varepsilon t}$ implies that $||u_{\varepsilon t}(\cdot,t+h)-u_{\varepsilon t}(\cdot,t) ||_{L^q(\Omega)} \to 0$ as $h\to 0$. Thus, $v_{\varepsilon t} \in C^0\big((0,T_{\max, \varepsilon}); W^{2,p}(\Omega)\big)$ for all $p>1$.
    \\\\
    Lastly, as in dimension one, for any $p > 1/2$, $W^{2,p}(\Omega) \hookrightarrow C^{0,\alpha}(\Omega)$, for $\alpha = 2-1/p$, we can finally conclude $v_{\varepsilon t} \in C^0\big(\Omega \times (0,T_{\max,\varepsilon})\big)$, entailing the desired $C^1$ regularity for $D_\varepsilon$.
    \\\\
    In this way, Lemma A.1 in \cite{TaoWinklerA1} becomes applicable, and a Moser iteration argument leads to
    $$\sup_{t\in(0,T_{\max,\varepsilon})} ||u_\varepsilon(\cdot,t)||_{L^\infty(\Omega)}<\infty,$$
    contradicting the hypothesis that $T_{\max,\varepsilon} < \infty$.
\end{proof}
Lastly, with all the previous results, a subsequence converging to the global weak solution of system \eqref{1.1} can be extracted by standard compactness arguments.
\begin{lemma}\label{l5.2}
    Let $p> 1$, $u_0$ and $f$ be such that \eqref{h1} and \eqref{h2} hold. Then, there exists a subsequence $(\varepsilon_j)_{j \in \mathbb{N}} \subset (0,1)$ with $\varepsilon_j \searrow 0$ as $j \to \infty$ and
    \begin{equation}\label{fin1}
    \begin{cases}
        u \in L^\infty_{\text{loc}}\big((0,\infty); L^r(\Omega)\big) ~ \text{ for all } r \geq 1, \\[1.5 ex]
        v \in C^0_\text{loc}(\bar{\Omega} \times [0,\infty)\big) \cap L^\infty_\text{loc}\big((0,\infty); C^{1,\alpha}(\Omega)\big),
    \end{cases}   
    \end{equation}
    with $u\geq 0$ and $v>0$ a.e. in $\Omega \times (0,\infty)$ satisfying
    \begin{align}    
         &u_{\varepsilon} \to u ~ \text{ a.e. in } \Omega \times (0,\infty) \text{ and in } L^q_{\text{loc}}\big(\bar{\Omega} \times [0,\infty)\big) \text{ for all } q \in [1,p)\\
         &v_\varepsilon \to v ~ \text{ a.e. in } \Omega \times (0,\infty) \text{ and in } C^0_\text{loc}(\bar{\Omega} \times [0,\infty)\big),\\
          &v_{\varepsilon x}   \overset{\ast}{\rightharpoonup} v_x ~\text{ in } L^\infty_{\text{loc}}\big(\Omega \times [0,\infty)\big),    
    \end{align}
    as $\varepsilon = \varepsilon_j$. Furthermore, the pair $(u,v)$ forms a global weak solution to system \eqref{1.1} in the sense of Definition \ref{weak-sol}.
\end{lemma}
\begin{proof}
    We begin by the convergence of $u_\varepsilon$. As previously mentioned, due to the lack of direct information regarding $u_{\varepsilon x}$, we consider the auxiliary sequence $z_\varepsilon := u_\varepsilon^{\frac{p+1}{2}}$ for $p\geq 2$. For any $T>0$, by combining Lemma \ref{l2.3} and Lemma \ref{l4.1} it follows that
    $$
    (z_\varepsilon)_{\varepsilon\in(0,1)}~ \text{ is bounded in } L^2\big((0,T); W^{1,2}(\Omega)\big),
    $$
    while by Lemma \ref{l4-aux}, we also have that
    $$
    (z_{\varepsilon t})_{\varepsilon\in(0,1)} ~ \text{ is bounded in } L^1\big((0,T); (W^{3,2}(\Omega))^*\big).
    $$
    Thus, using the Aubin-Lions lemma \cite{aubin-lions}, there exists $(\varepsilon_{j_1})_{{j_1}\in \mathbb{N}} \subset (0,1)$, with $\varepsilon_{j_1} \to 0$ as $j_1 \to \infty$ and a function 
    $$z \in L^2_\text{loc}\big([0,\infty); W^{1,2}(\Omega)\big),$$
    such that
    \begin{equation}\label{5.4}
        z_{\varepsilon} \to z ~ \text{ a.e. in } \Omega \times (0,\infty) \text{ and in } L^2_{\text{loc}}\big(\bar{\Omega} \times [0,\infty)\big),
    \end{equation}
    as $\varepsilon = \varepsilon_{j_1}$. Therefore, as for all $p \geq 2$ the map $s \mapsto s^\frac{2}{p+1}$ is continuous and monotone, we can define 
    $$u := z^{\frac{2}{p+1}} \in L^\infty_\text{loc}\big((0,\infty); L^r(\Omega)\big), \quad \text{for all } r \geq 1,$$
    such that 
    $$
    u_\varepsilon \to u ~ \text{ a.e. in } \Omega \times (0,\infty) \text{ and in } L^q_{\text{loc}}\big(\bar{\Omega} \times [0,\infty)\big) \text{ for all } q \in [1,p),
    $$
    for $\varepsilon = \varepsilon_{j_1}$, as a consequence of the Vitali convergence theorem.
    \\\\
    For $v_\varepsilon$ we employ the Arzelà-Ascoli theorem. From Lemma \ref{l3.3}, we know that
    $$
    (v_\varepsilon)_{\varepsilon \in (0,1)} ~\text{ is bounded in } L^\infty\big(\Omega \times [0,T)\big),
    $$
    for any $T>0$. Moreover, the bound for $||v_{\varepsilon x}||_{L^\infty(\Omega)}$ from Lemma \ref{l3.4} gives uniform Lipschitz continuity in space, while Lemma \ref{l4.5} provides uniform Hölder continuity in time. Thus, the sequence $(v_\varepsilon)_{\varepsilon \in (0,1)}$ is uniformly equicontinuous, and from the Arzelà-Ascoli theorem we have
    $$
    (v_\varepsilon)_{\varepsilon \in (0,1)} ~\text{ is relatively compact in } C^0_\text{loc}\big(\bar{\Omega} \times (0,\infty)\big).
    $$
    Moreover, from Lemma \ref{l3.4}, we also have boundedness of $v_\varepsilon$ in $L^\infty_\text{loc}\big((0,\infty); C^{1,\alpha}(\Omega)\big)$. As these bounds are preserved when restricted to the previous subsequence $(\varepsilon_{j_1})_{j_1 \in \mathbb{N}}$, a second subsequence $(\varepsilon_{j_2})_{{j_2}\in \mathbb{N}} \subset (\varepsilon_{j_1})_{{j_1}\in \mathbb{N}} \subset (0,1)$, with $\varepsilon_{j_2} \to 0$ as $j_2 \to \infty$ can be extracted, as well as a function 
    $$
    v \in C^0_\text{loc}(\bar{\Omega} \times [0,\infty)\big) \cap L^\infty_\text{loc}\big((0,\infty); C^{1,\alpha}(\Omega)\big),
    $$
    such that
    \begin{equation}\label{5.6}
        v_\varepsilon \to v ~ \text{ a.e. in } \Omega \times (0,\infty) \text{ and in } C^0_\text{loc}(\bar{\Omega} \times [0,\infty)\big),
    \end{equation}
    for $\varepsilon = \varepsilon_{j_2}$. Correspondingly, another application of the Banach-Alaoglu theorem entails
    \begin{equation}\label{5.7}
         v_{\varepsilon x}   \overset{\ast}{\rightharpoonup} v_x ~\text{ in } L^\infty_{\text{loc}}\big(\Omega \times [0,\infty)\big),
    \end{equation}
    for $\varepsilon = \varepsilon_{j_2}$.
    \\\\
    To finish the proof, we finally check that $(u,v)$ in fact form a weak solution to system \eqref{1.1} in the sense of Definition \ref{weak-sol}. From \eqref{fin1}, it is direct to see that \eqref{2.01} and \eqref{2.02} hold. For the weak formulation, as for every $\varepsilon \in (0,1)$, $(u_\varepsilon, v_\varepsilon)$ are the unique classical solution to the regularized system \eqref{2.1}, in particular for any $\varphi \in C_0^\infty (\bar{\Omega} \times [0, \infty))$ they satisfy
    \begin{equation*}
    \begin{split} 
     \int_0^\infty \int_\Omega u_\varepsilon \varphi_t + \int_\Omega (u_0 +\varepsilon)\varphi(\cdot, 0) =  -\frac{1}{2}\int_0^\infty \int_\Omega u_\varepsilon^2 v_{\varepsilon x} \varphi_x - \frac{1}{2}\int_0^\infty \int_\Omega u_\varepsilon^2 v_\varepsilon \varphi_{xx} \\[1.5ex]
     - \int_0^\infty \int_\Omega u_\varepsilon^2 v_\varepsilon v_{\varepsilon x} \varphi_x  - \int_0^\infty \int_\Omega u_\varepsilon v_\varepsilon \varphi,
\end{split}
\end{equation*}
as well as
\begin{equation*}
    \int_\Omega v_{\varepsilon x}(\cdot,t) \varphi_x(\cdot,t) +\int_\Omega u_\varepsilon(\cdot,t) v_\varepsilon(\cdot,t) \varphi(\cdot,t) = \int_\Omega f(\cdot,t) \varphi(\cdot,t), \quad \text{for a.e. } t >0.
\end{equation*}
Therefore, by passing to the limit with $\varepsilon = \varepsilon_{j_2} \searrow 0$, using the convergence established in \eqref{5.4}-\eqref{5.7} we deduce that $(u,v)$ form a global weak solution to system \eqref{1.1} in the considered sense.
\end{proof}
Proof of Theorem \ref{t1}. The result directly follows from Lemma \ref{l5.2}.

\section*{Acknowledgments}
This work was partially supported by Grant FPU23/03170 from the Spanish Ministry of Science, Innovation and Universities, as well as by a DAAD Research Grant - Bi-nationally Supervised Doctoral Degrees/Cotutelle, 2024/25 (57693451), Grant number 91907995 (F.H.-H.)


\begin{thebibliography}{1}

\bibitem{Bellomo} N. Bellomo, A. Bellouquid, Y. Tao, M. Winkler. \textit{Toward a mathematical theory of Keller-Segel models of pattern formation in biological tissues.} {Math. Model. Methods Appl. Sci.}, 25, 1663–1763, (2015).

\bibitem{BudreneBerg} E.O. Budrene, H.C. Berg. \textit{Complex patterns formed by motile cells of Escherichia coli.} {Nature}, 349, (1991).

\bibitem{fractal2} E. Ben-Jacob, O. Schochet, A. Tenenbaum, I. Cohen, A. Czirók, T. Vicsek. \textit{Generic modelling of cooperative growth patterns in bacterial colonies}. Nature, 368 46–49 (1994).

\bibitem{gilbard-trudinger} D. Gilbard, N.S. Trudinger. \textit{Elliptic Partial Differential Equations of Second Order}. Springer Classics in Mathematics, (1998).

\bibitem{Fujikawa89} H. Fujikawa, M. Matsushita. \textit{Fractal growth of Bacillus subtilis on agar plates.} {J. Phys. Soc. Japan}, 47, (1989).

\bibitem{Fujikawa92} H. Fujikawa. \textit{Periodic growth of Bacillus subtilis colonies on agar plates.} {Physica A}, 189, (1992).

\bibitem{mah} M.A. Herrero, J.L.L. Velázquez. \textit{A blow-up mechanism for a chemotaxis model.} {Ann. Sc. Norm. Super. Pisa Cl. Sci.}, 24, (1997).

\bibitem{HH-Cauchy} F. Herrero-Hervás. \textit{Global weak solutions to a doubly degenerate nutrient taxis system on the whole real line.} Preprint, 
https://doi.org/10.48550/arXiv.2508.07503, (2025).

\bibitem{Hillen} T. Hillen, K.J. Painter. \textit{A users guide to PDE models for chemotaxis}. {Journal of Mathematical Biology}, 58, 183–217, (2009).
    
\bibitem{Horstmann} D. Horstmann.  \textit{From 1970 until present: the Keller-Segel model in chemotaxis and its consequences}, Jahresbericht der Deutschen Mathematiker-Vereinigung. 105 (3),  103-165, (2003).

\bibitem{Kawasaki} K. Kawasaki, A. Mochizuki, M. Matsushita, T. Umeda, N. Shigesada. \textit{Modeling spatio-temporal patterns generated by Bacillus subtilis.} {J. Theor. Biol.}, 188, (1997).

\bibitem{KS1} E.F. Keller, L.A  Segel \textit{Initiation of slime mold aggregation viewed as an instability}. J. Theoret. Biol. {\bf{26}},  399-415, (1970).
	
\bibitem{KS2} E.F. Keller, L.A  Segel \textit{A model for chemotaxis.} J. Theoret. Biol. 30,   225-234, (1971).

\bibitem{li-wink} G. Li, M. Winkler. \textit{Nonnegative solutions to a doubly degenerate nutrient taxis system}. Commun. Pure Appl. Anal,  21 (2), 687-704 (2022)

\bibitem{winkler-li-logist} G. Li, M. Winkler. \textit{Continuous solutions for a two-dimensional cross-diffusion problem involving doubly degenerate diffusion and logistic proliferation}. Analysis and Applications, 23(4), 489-510 (2025).

\bibitem{plaza13} J.F. Leyva, C. Málaga, R.G. Plaza. \textit{The effects of nutrient chemotaxis on bacterial aggregation patterns
with non-linear degenerate cross diffusion.} Physica A 392, (2013).

\bibitem{Fujikawa90} M. Matsushita, H. Fujikawa. \textit{Diffusion-limited growth in bacterial colony formation.} {Physica A}, 168, (1990).

\bibitem{fractal1} T. Matsuyama, M. Matsushita. \textit{Fractal morphogenesis by a bacterial cell population}. Crit. Rev. Microbiol., 19 (2) 117–135 (1993).

\bibitem{logistic} X. Pan. \textit{Superlinear degradation in a doubly degenerate nutrient taxis system.} Nonlinear Anal.-Real World Appl. 77, 104040 (2024)

\bibitem{Tao11} Y. Tao. \textit{Boundedness in a chemotaxis model with oxygen consumption by bacteria.} {J. Math. Anal. Appl.}, 381 (2), (2011).

\bibitem{TaoWinkler12} Y. Tao, M. Winkler. \textit{Eventual smoothness and stabilization of large-data solutions in a three-dimensional chemotaxis system with consumption of chemoattractant.} {J. Differ. Equ.}, 252 (3), (2012).

\bibitem{TaoWinklerA1} Y. Tao, M. Winkler. \textit{Boundedness in a quasilinear parabolic-parabolic Keller-Segel system with subcritical sensitivity.} {J. Differ. Equ.}, 252(1), (2012).

\bibitem{TaoWinkler2019} Y. Tao, M. Winkler. \textit{Global smooth solvability of a parabolic-elliptic nutrient taxis system in domains of arbitrary dimension.} {J. Differ. Equ.}, 267, (2019).

\bibitem{aubin-lions} R. Temam, \textit{Navier-Stokes Equations. Theory and Numerical Analysis}. Stud. Math. Appl., Vol. 2, North-Holland, Amsterdam (1977)

\bibitem{winkler1d} M. Winkler. \textit{Does spatial homogeneity ultimately prevail in nutrient taxis systems? A paradigm for structure support by rapid diffusion decay in an autonomous parabolic flow}. Trans. Amer. Math. Soc. 374, 219-268 (2021).

\bibitem{winkler2d} M. Winkler. \textit{Small-signal solutions of a two-dimensional doubly degenerate taxis system modeling bacterial motion in nutrient-poor environments.} Nonlinear Anal.-Real World Appl. 63, 103407 (2022).

\bibitem{winkler-l-inf-2d} M. Winkler. \textit{$L^\infty$ bounds in a two-dimensional doubly degenerate nutrient taxis system with general cross-diffusive flux}. J. Differential Eq. 400, 423-456 (2024).

\bibitem{yuting} Y. Xiang. \textit{Boundedness and large-time behavior in a two-species doubly degenerate diffusion chemotaxis system with logistic proliferation} Nonlinear Anal.-Real World Appl. 89, 104525 (2026).

\bibitem{Wu16} S. Wu, J. Shi, B. Wu. \textit{Global existence of solutions and uniform persistence of a diffusive predator--prey model with prey-taxis.} {J. Differ. Equ.}, 260 (7), (2016).

\bibitem{Duan} D. Wu. \textit{The asymptotic behavior of solutions to a doubly degenerate chemotaxis-consumption system in the two-dimensional setting.} Preprint 
https://doi.org/10.48550/arXiv.2409.12083 (2024)

\bibitem{ZhangLi} Q. Zhang, Y. Li. \textit{Stabilization and convergence rate in a chemotaxis system with consumption of chemoattractant.} {J. Math. Phys.}, 56(8), (2015).

\bibitem{ZL1} Z. Zhang, Y. Li. \textit{Boundedness in a two-dimensional doubly degenerate nutrient taxis system}. Math. Meth. Appl. Sci. Online ready doi.org/10.1142/S0218202526500077 (2026)


\end{thebibliography}
\end{document}